\documentclass[11pt]{article}

\usepackage[english]{babel}
\usepackage{amsmath,amsfonts,amssymb,amsthm}
\usepackage{latexsym}
\usepackage{makeidx}
\usepackage{amscd}

\usepackage{vmargin}
\setmarginsrb{10mm}{20mm}{10mm}{20mm}
             {8mm}{8mm}{8mm}{8mm}

\numberwithin{equation}{section}

\newtheorem{definition}{Definition}[section]
\newtheorem{lemma}[definition]{Lemma}
\newtheorem{theorem}[definition]{Theorem}
\newtheorem{corollary}[definition]{Corollary}
\newtheorem{em-discussion}[definition]{Discussion}

\newtheorem{proposition}[definition]{Proposition}
\newtheorem{fact}[definition]{Fact}
\newtheorem{em-example}[definition]{Example}
\newtheorem{em-def}[definition]{Definition}        
\newtheorem{em-remark}[definition]{Remark}         
\newtheorem{em-question}[definition]{Question}

\newtheorem{problem}[definition]{Problem}

\newenvironment{example}{\begin{em-example} \em }{ \end{em-example}}
\newenvironment{remark}{\begin{em-remark} \em }{\end{em-remark}}

\newenvironment{discussion}{\begin{em-discussion} \em }{ \end{em-discussion}}

\newcommand{\R}{\mathbb R}
\newcommand{\N}{\mathbb N}

\newcommand{\Q}{\mathbb Q}
\newcommand{\Z}{\mathbb Z}

\newcommand{\f}{\phi}
\def\abg{\mathbf{AbGrp}}
\def\af{\mathbf{Flow}}

\def\mod{\mathbf{Mod}}
\def\bim{\mod_{R[t]}}

\def\mod{\mathbf{Mod}}
\def\ent{\mathrm{ent}}
\def\End{\mathrm{End}}
\def\Hom{\mathrm{Hom}}

\def\Fin{\mathrm{Fin}}

\def\P{\mathbf{P}}
\def\QQ{\mathfrak Q}

\def\H{\mathfrak{H}}
\def\HB{\mathfrak H_b}
\def\r{\mathbf r}
\global\def\hull#1{\langle{#1}\rangle}

\input xy
\xyoption{all}

\title{Entropy in a category}

\author{Dikran Dikranjan
\\{\footnotesize {\tt  dikran.dikranjan@uniud.it}} 
\\{\footnotesize Dipartimento di Matematica e Informatica,}
\\{\footnotesize Universit\`{a} di Udine,}
\\{\footnotesize Via delle Scienze, 206 - 33100 Udine, Italy} 
 \and Anna Giordano Bruno
\\{\footnotesize {\tt  anna.giordanobruno@uniud.it}} 
\\{\footnotesize Dipartimento di Matematica e Informatica,}
\\{\footnotesize Universit\`{a} di Udine,}
\\{\footnotesize Via delle Scienze, 206 - 33100 Udine, Italy} 
 }

\date{Dedicated to the memory of Maria Silvia Lucido}

\begin{document}

\maketitle


\abstract{The Pinsker subgroup of an abelian group with respect to an endomorphism was introduced in the context of algebraic entropy. Motivated by the nice properties and characterizations of the Pinsker subgroup, we generalize its construction in two directions. We introduce the concept of entropy function $h$ of an abelian category and define the Pinsker radical with respect to $h$, so that the class of all objects with trivial Pinsker radical is the torsion class of a torsion theory.}

\section{Introduction}

The concept of entropy was invented by Clausius in Physics in 1865 and carried over to Information Theory by Shannon in 1948, to Ergodic Theory by Kolmogorov and Sinai in 1958, and to Topological Dynamics by Adler, Konheim and McAndrew in 1965 \cite{AKM} (see Section \ref{contr-sec} for the definitions of measure entropy and topological entropy).

\smallskip
In the context of abelian groups, the algebraic entropy $\ent$ of endomorphisms $\phi$ of abelian groups $G$ was introduced first by Adler, Konheim and McAndrew \cite{AKM}, and later on by Weiss \cite{W}, using trajectories of finite subgroups $F$ of $G$ with respect to $\phi$ (see Section \ref{entW-sec}). So the algebraic entropy $\ent$ is appropriate for endomorphisms of torsion abelian groups. More precisely, the value of $\ent$ for an endomorphism of an abelian group coincides with the value of $\ent$ for the restriction of the endomorphism to the torsion part of the group, and so it is trivially zero for endomorphisms of torsion-free abelian groups.

Peters \cite{Pet} modified the definition of algebraic entropy for automorphisms of arbitrary abelian groups, using non-empty finite subsets instead of finite subgroups. In \cite{DG} this notion was extended to endomorphisms of abelian groups (see Section  \ref{ent-sec}). We denote the algebraic entropy defined in this way by $h_a$.

Another kind of algebraic entropy generalizing $\ent$ was introduced in \cite{SZ}, namely, the $i$-entropy $\ent_i$ defined for endomorphisms of modules with respect to an additive invariant $i$ (see Section \ref{ient-sec}).
 
\medskip
In the framework of Ergodic Theory, the \emph{Pinsker $\sigma$-algebra} $\mathfrak P(\phi)$ of a measure preserving transformation $\phi$ of a measure space $(X, {\mathcal B}, \mu)$ is defined as the maximum $\sigma$-subalgebra of ${\mathcal B}$ such that $\phi$ restricted to $(X, \mathfrak P(\phi),\mu\restriction_{{\mathcal B}})$ has measure entropy zero. 

A similar concept was introduced in Topological Dynamics as follows. A \emph{topological flow} is a pair $(X,\phi)$, where $X$ is a compact Hausdorff space and $\phi:X\to X$ a homeomorphism. Moreover, a \emph{factor} $(\pi,(Y,\psi))$ of $(X,\phi)$ is a topological flow $(Y,\psi)$ together with a continuous surjective map $\pi:X\to Y$ such that $\pi\circ \phi=\psi\circ\pi$. A topological flow $(X,\phi)$ admits a greatest factor of zero topological entropy, called \emph{topological Pinsker factor} \cite{BL} (see also \cite{KerLi}).

\smallskip
The counterpart of these notions for the algebraic entropy was introduced and studied in its various aspects in \cite{DG1}. For an abelian group $G$ and an endomorphism $\phi$ of $G$, the \emph{Pinsker subgroup} is the greatest $\phi$-invariant subgroup of $G$ where the restriction of $\phi$ has zero algebraic entropy $h_a$.

\bigskip
The aim of this paper is two-fold. On one hand, we address the category-minded reader with a blend of results showing the necessity to develop a rigorous categorical approach to entropy. On the other hand, we generalize the construction of the Pinsker subgroup from \cite{DG1} in several directions.

First, we replace the category of abelian groups by abelian categories. 
Second, we introduce a general (abstract) notion of entropy function for abelian categories in Definition \ref{h-def}. We impose only three very mild axioms (A1), (A2) and (A3), and we show that they form a minimal set of properties sufficient to carry out the construction of the  Pinsker radical (see Remark \ref{minimal}). The essence of our approach is to make clear that many of the results on the already defined entropies can be proved for abstractly defined entropy functions, without any recourse to the specific formulas defining the known entropies. This should be compared with the totally opposite approach in  \cite{DG1}, where no use was made of (A2), but the specific features of the algebraic entropy $h_a$ were heavily used to establish the polynomial vs exponential growth dichotomy. 
Third, in the context of module categories we introduce several radicals capturing the dynamical behavior of module endomorphisms and ``approximating" the Pinsker submodule. These radicals are used to develop a larger set of axioms with the aim to determine uniquely the entropy function. 

\medskip
A categorical approach to entropy from a completely different point of view is given in \cite{DS} (se also \cite{DBcn}).

\subsection{Main results}

In the sequel $\mathfrak M$ will be a well-powered cocomplete abelian category. 

\begin{definition}\label{h-def}
An \emph{entropy function} $h$ of $\mathfrak M$ is a function $h:\mathfrak M\to \R_+\cup\{\infty\}$ such that:
\begin{itemize}
\item[\emph{(A1)}] $h(0) = 0$ and $h(M)=h(N)$ if $M$ and $N$ are isomorphic objects in $\mathfrak M$;
\item[\emph{(A2)}] $h(M)=0$ if and only if $h(N)=0=h(Q)$ for every exact sequence $0\longrightarrow N\longrightarrow M \longrightarrow Q\longrightarrow0$ in $\mathfrak M$;
\item[\emph{(A3)}] for a set $\{M_j:j\in I\}$ of objects of $\mathfrak M$, $h(\bigoplus_{j\in J} M_j) = 0$ if and only if $h(M_j) = 0$ for all $j\in J$.
\end{itemize}
An entropy function $h$ of $\mathfrak M$ is \emph{binary} if it takes only the values $0$ and $\infty$.
\end{definition}

\smallskip 
A function $h$ with (A1) is called (for obvious reasons) an \emph{invariant} of $\mathfrak M$. 
In the case $\mathfrak M=\mod_A$ is the category of left $A$-modules over a ring $A$, the axiom (A3), in the presence of (A1) and (A2), is equivalent to: if $M\in\mod_A$ and $M$ is direct limit of its submodules $\{M_j:j\in J\}$, then $h(M) = 0$ if and only if $h(M_j) = 0$ for every $j\in J$.
Moreover, the length functions in the sense of \cite{SVV,V} (see Definition \ref{L-def}) are special entropy functions of $\mod_A$, so our approach generalizes also this known notion.



\bigskip
Section \ref{ef-sec} is dedicated to entropy functions of $\mathfrak M$. In Section \ref{lattice-sec} we 
define a preorder $\prec$ of the class $\H(\mathfrak M)$ of all entropy functions of $\mathfrak M$ induced by the order of $\R_+\cup\{\infty\}$. 
It makes $(\H(\mathfrak M),\prec)$ a complete lattice as well as its sublattice $(\HB(\mathfrak M),\prec)$ of all binary entropy functions of $\mathfrak M$ (see Proposition \ref{H-cl} and Corollary \ref{Hb-cl}, respectively).

In Section \ref{p-sec} we define the \emph{Pinsker radical} $\P_h:\mathfrak M\to \mathfrak M$ with respect to an entropy function $h$ of $\mathfrak M$, and we prove that it is a hereditary radical in Theorem \ref{TT}. 
This is a counterpart of the Pinsker subgroup, defined for the algebraic entropy $h_a$ in \cite{DG1}, which was the motivating point of the present paper.

It seems then natural to investigate the subclass of $\mathfrak M$ consisting of all objects in $\mathfrak M$ where a given entropy function $h$ of $\mathfrak M$ takes value zero. So let $$\mathcal T_h=\{M\in\mathfrak M: h(M)=0\}.$$ 

In analogy to \cite{DG1}, for a given entropy function $h$ of $\mathfrak M$, we say that $M$ has \emph{completely positive entropy} if $h(N)>0$ for every non-zero subobject $N$ of $M$, and we denote this by $h(M)>\!\!>0$.
As a natural counterpart of $\mathcal T_h$, we define the class of all objects in $\mathfrak M$ with completely positive entropy, that is, $$\mathcal F_h=\{M\in\mathfrak M: h(M)>\!\!>0\}.$$
Section \ref{tt-sec} is dedicated to the torsion theory $\mathfrak t_h=(\mathcal T_h,\mathcal F_h)$ relative to the Pinsker radical $\P_h$ of $\mathfrak M$. 
Since $\P_h$ is a hereditary radical, $\mathfrak t_h$ is a hereditary torsion theory in $\mathfrak M$. 

Moreover, Theorem \ref{hb<->tt} shows that the assignment $h\mapsto \mathfrak t_h$ is a bijective order preserving correspondence between binary entropy functions of $\mathfrak M$ and hereditary torsion theories in $\mathfrak M$. So, there may be information in an entropy function which is not captured by the hereditary torsion theory, and binary entropy functions are simply those which do not contain any additional information.

\medskip
In Section \ref{ef-af-sec} we restrict to the fundamental case of entropy functions of module categories.

\smallskip
We start recalling in Section \ref{flow-sec} the definition of the category $\af_\mathfrak X$ of flows of a category $\mathfrak X$. 
Prominent examples will be the category $\abg$ of all abelian groups, and the category $\mod_R$ of all left modules over a ring $R$, (or, more generally, an abelian category $\mathfrak X=\mathfrak M$).
In the latter case,  we denote $\af_{\mod_R}$ simply by $\af_R$, and we call \emph{algebraic flow} an object $(M,\phi)$ of $\af_R$.
Theorem \ref{iso} shows that 
\begin{equation}\label{af=mod}
\af_R \cong \mod_{R[t]}.
\end{equation}
So every $R[t]$-module can be considered as an algebraic flow, and viceversa. 
In particular, an entropy function $h$ can be viewed equivalently as an entropy function of $\af_R$ or as an entropy function of $\mod_{R[t]}$.

\smallskip
The general Definition \ref{h-def}  of entropy function allows us to consider entropy functions in \emph{arbitrary} module categories, going ``beyond the limits imposed by endomorphisms'' in the leading example $\af_R$. Nevertheless, we shall very often turn back to this case, which is the true source of the definition of entropy function. Indeed, in $\af_R$ \emph{compositions} (and in particular, \emph{powers}) of endomorphisms are available, and they have no counterpart in the general case.

\smallskip
In Section \ref{ax-sec} we discuss a collection of axioms to add to those of Definition \ref{h-def} in order to have entropy functions $h$ of $\af_R$ with a behavior closer to the original dynamical nature of this notion. 
Indeed, most of these additional axioms are very simple and natural, as for example
\begin{itemize}
\item[(A0)] $h(0_M)=0$ and $h(1_M)=0$ for every $M\in\mod_R$.
\end{itemize}
Other axioms (that we call (A2$^*$), (A4), (A4$^*$) and (A5)) imitate the properties that give uniqueness in the particular case of the algebraic entropy $h_a$ (see Theorem \ref{UT} below). However, we leave open the problem of finding a family of axioms giving uniqueness for an abstractly defined entropy function $h$ of $\af_R$.

\medskip
In Section \ref{qp-sec} we introduce radicals of the category $\af_R$ capturing the dynamics of the endomorphisms. These radicals do not depend on any specific entropy function, but can be compared with the Pinsker radical of an entropy function of $\af_R$.
The radical $\QQ$, inspired by a characterization of the Pinsker subgroup given in \cite{DG1}, is defined using the quasi-periodic points. The radicals $\mathfrak O$ and $\mathfrak I$, generated by all  zero endomorphisms and all identities respectively, satisfy $\mathfrak O\leq\QQ$ and $\mathfrak I\leq\QQ$, and provide a flexible language to describe the axiom (A0).

Motivated by the fact that in $\af_\abg$ the radical $\QQ$ coincides with the Pinsker radical $\P_{h_a}$, 
we prove that $\QQ\leq \P_h$ holds for any entropy function $h$ of $\af_R$ satisfying (A0) and (A4$^*$) (see Theorem \ref{semiA6}). 

We introduce other two radicals, namely $\mathfrak A \leq \mathfrak W$, that correspond respectively to the notion of pointwise integral and that of pointwise algebraic endomorphism. They both contain $\QQ$ and provide a better approximation from below and from above of the Pinsker radical (see Corollary \ref{PvsW}, and see also  
Theorem \ref{A<P} for a particular case of entropy function). On the other hand, Example \ref{example} shows that the Pinsker radical may fail to coincide with  $\mathfrak A$ (hence with  $\mathfrak Q$ either). 

\medskip
Section \ref{ex-sec} is dedicated to the specific known algebraic entropies, which are examples of entropy functions in the sense of Definition \ref{h-def}.
In Section \ref{ent-sec} we go back to our motivating example, that is the algebraic entropy $h_a$ of $\af_\abg$. In this particular case \eqref{af=mod} gives $\af_\abg\cong\mod_{\Z[t]}$. So $h_a$ can be viewed also as an entropy function of $\mod_{\Z[t]}$.
Section \ref{entW-sec} is dedicated to the entropy function that is best understood so far, namely the algebraic entropy $\ent$ of $\af_{\mathbf{TorAbGrp}}$, 
where $\mathbf{TorAbGrp}$ is the category of all torsion abelian groups. 
In Section \ref{ient-sec} we consider the $i$-entropy, introduced in \cite{SZ} for module categories $\mod_R$ over a ring $R$ and additive invariants $i$ of $\mod_R$, and developed in \cite{SVV,Simone}. 
Applying the results of Section \ref{ef-sec} to these particular cases, we find a torsion theory with respect to each of the considered entropy functions.
Moreover, we show that the Pinsker radical satisfies $\mathfrak A \leq \P_{\ent_i} \leq \mathfrak W$ when the ring $R$ is an integral domain (see Theorem \ref{A<P})
and actually coincides with the radical $\mathfrak W$ under appropriate conditions (see Corollary \ref{LASTcorollary}). 

\medskip
Finally, in Section \ref{contr-sec} we consider the measure entropy and the topological entropy. They satisfy the same properties with respect to (A1) and (A2) of Definition \ref{h-def}, but they are continuous under taking inverse limits. So we give an idea on how it could be possible to proceed in this different situation and leave open the problem to treat these that we call ``contravariant entropy functions''.

\bigskip
Parts of these results, in preliminary form, were exposed at seminar talks by the first named author at the Seminar of Dynamical Systems at the Hebrew University of Jerusalem,  the Mathematical Colloquium of Bar Ilan University of Tel Aviv and the Seminar of Category Theory at Coimbra University in the autumn of 2007. It is a pleasure to thank the participants of these seminars, as well as Luigi Salce and Simone Virili (who gave us copies of preliminary versions of \cite{SVV,Simone} in May 2010) for useful comments. Last but not least, thanks are due also to Peter Vamos, who kindly sent his paper \cite{V} to the first named author in the autumn of 2007 (unfortunately, we realized the full power of Vamos' ideas only after reading \cite{SVV}). 
 

\subsubsection*{Notation and terminology}

We denote by $\mathbb Z$, $\mathbb N$, $\mathbb N_+$, $\Q$ and $\R$ respectively the set of integers, the set of natural numbers, the set of positive integers, the set of rationals and the set of reals. Moreover, $\R_+=\{r\in\R:r\geq0\}$. For $m\in\mathbb N_+$, we use $\mathbb Z(m)$ for the finite cyclic group of order $m$. The free-rank of an abelian group $G$ is denoted by $r_0(G)$.

Let $R$ be a ring. We denote by $R[t]$ the ring of polynomials with coefficients in $R$. We indicated with $\mod_R$ the category of left $R$-modules. For $M\in\mod_R$, the submodule of torsion elements of $M$ is $t(M)$, while $\End(M)$ is the ring of all endomorphisms of $M$.

 For an abelian category $\mathfrak M$ we write $M\in\mathfrak M$ if $M$ is an object of $\mathfrak M$ and $N\subseteq M$ if $N$ is a subobject of $M$. 
For $M\in\mathfrak M$, we denote by $0_M$ the zero morphism of $M$ and by $1_M$ the identity morphism of $M$.
Moreover, for $M_1,M_2\subseteq M$, we denote by $M_1+M_2$ the join of $M_1$ and $M_2$ and by $M_1\cap M_2$ the intersection of $M_1$ and $M_2$. If $N$ is a subobject of $M$, $M/N$ is the quotient object. For a morphism $f:M\to N$ we denote by $f(M)$ the image of $f$, which is a subobject of $N$. Moreover, if $i:K\to M$ is a subobject of $M$, $f(K)$ stands for the image of $f\circ i$, which is a subobject of $N$.

For a family $\{M_j:j\in J\}$ of objects of $\mathfrak M$, we denote by $\bigoplus_{i\in J}M_j$ the coproduct and by $\prod_{j\in J}M_j$ the product, if they exist. In particular, for a cardinal $\alpha$ we denote by $M^{(\alpha)}$ the coproduct $\bigoplus_\alpha M$ (and by $M^\alpha$ the product $\prod_\alpha M$) of $\alpha$ many copies of $M$.

If $\mathfrak M$ is cocomplete, the join $\sum_{j\in J}M_j$ of a family $\{M_j:j\in J\}$ of subobjects of $M\in\mathfrak M$, is the image $f(\bigoplus_{j\in J}M_j)$ of the coproduct $\bigoplus_{j\in J}M_j$
under the canonical morphism $f:\bigoplus_{j\in J}M_j\to M$. 

\section{The Pinsker torsion theory}\label{ef-sec}

\subsection{The lattice of entropy functions}\label{lattice-sec}

As imposed in the introduction, $\mathfrak M$ will be a well-powered cocomplete abelian category in the sequel.

\begin{discussion}\label{A3}
\begin{itemize}
\item[(a)] The axiom {(A2)} implies that $h(M_1\oplus\ldots\oplus M_n)=0$ if and only if $h(M_1)=\ldots=h(M_n)=0$ for $M_1,\ldots,M_n\in\mathfrak M$. 
\item[(b)] If $M\in\mathfrak M$ and $M_1,\ldots,M_n$ are subobjects of $M$, then $h(M_1+\ldots+M_n)=0$ if $h(M_1)=\ldots=h(M_n)=0$. This follows from (a) and {(A2)}, since $M_1+\ldots+M_n$ is a quotient of $M_1\oplus\ldots\oplus M_n$.
\item[(c)] Item (b) can be stated in the following more general form, which is equivalent to (A3) in the presence of {(A1)} and {(A2)}: for a set $\{M_j:j\in J\}$ of objects of $M$, $h(\sum_{j\in J}M_j)=0$ if and only if $h(M_j)=0$ for all $j\in J$.
\end{itemize}
\end{discussion}

\begin{remark}\label{A3d}
For a ring $A$ and $\mod_A$, the axiom (A3) holds precisely when, for $M$ direct limit of its submodules $\{M_j:j\in J\}$, $h(M)=0$ if and only if $h(M_j)=0$ for all $j\in J$.

For $M\in\mod_A$, let $\mathcal F(M)$ be the family of all finitely generated submodules of $M$.
By the previous part of the remark, the axiom (A3) is equivalent also to: $h(M)=0$ if and only if $h(N)=0$ for every $N\in\mathcal F(M)$.

Moreover, $h(A) = 0$ yields that $h \equiv 0$ because of Discussion \ref{A3}(c). 
\end{remark}

The order $\leq$ in $\R_+ \cup \{\infty\}$ (with $x< \infty$ for all $r\in \R_+$) defines a partial order $\prec$ in the class $\H(\mathfrak M)$ of all entropy functions of $\mathfrak M$ by letting
$$
h_1 \prec h_2\ \text{if and only if}\ h_1(M) \leq h_2(M)\ \text{for every}\ M\in\mathfrak M.
$$
For $h_1, h_2 \in \H(\mathfrak M)$ define $h_1 + h_2$ by $(h_1+ h_2)(M) = h_1 (M)+ h_2(M)$ for every $M\in\mathfrak M$. 
Then $h_1 + h_2 \in \H(\mathfrak M)$. In particular, this defines the multiples $nh \in \H(\mathfrak M)$
for $h\in \H(\mathfrak M)$ and $n \in \N_+$ (as usual, we agree that $x + \infty=\infty+ x = \infty$ for all 
$x\in \R_+\cup\{\infty\}$; in particular, $n\infty = \infty$ for $n\in \N_+$). 

\begin{proposition}\label{H-cl}
The pair $(\H(\mathfrak M),\prec)$ is a (large) complete lattice.
\end{proposition}
\begin{proof}
Since the constant zero is the bottom element of $\H(\mathfrak M)$, to show that $(\H(\mathfrak M),\prec)$ is a complete lattice it suffices to verify that there exist arbitrary suprema. So, for a class $\{h_j\in\H(\mathfrak M): j\in J\}$, let $h=\sup_{j\in J} h_j$ be defined by $h(M) = \sup_{j\in J} h_j(M)$ for every $M\in\mathfrak M$. It is easy to see that $h$ is still in $\H(\mathfrak M)$. 
\end{proof}

Let $$\HB(\mathfrak M) = \{h\in \H(\mathfrak M): h\ \text{is binary}\}.$$

\begin{corollary}\label{Hb-cl}
The sublattice $(\HB(\mathfrak M),\prec)$ of $(\H(\mathfrak M),\prec)$ is a complete lattice.
\end{corollary}
\begin{proof}
It is easy to see that $\HB(\mathfrak M)$ is stable under taking suprema and infima in $\H(\mathfrak M)$.
\end{proof}
 
This allows us to give the following

\begin{definition}
The \emph{binary hull} of $h\in\H(\mathfrak M)$ is the smallest $h^b\in \HB(\mathfrak M)$ above $h$, i.e., $$h^b = \inf \{h'\in \HB(\mathfrak M): h'\succ h\}.$$
\end{definition}

One can ``approximate" the binary hull $h^b$ also from below, as the next proposition shows: 

\begin{proposition}\label{bheq}
Let $h\in\H(\mathfrak M)$. Then the following conditions are equivalent: 
\begin{itemize}
  \item[(a)] $h\in\H_b(\mathfrak M)$;
  \item[(b)] $h= nh$ for every $n \in \N_+$;
  \item[(c)] there exists $n >1$ such that $h= nh$.
\end{itemize}
Consequently, $h^b = \sup \{nh: n\in \N\}.$ In particular, for every $M\in\mathfrak M$, 
\begin{equation}\label{h*}
h^b(M)=\begin{cases}0 & \text{if and only if}\ h(M)=0,\\ \infty & \text{otherwise}.\end{cases}
\end{equation}
\end{proposition}
\begin{proof} 
(a)$\Rightarrow$(b) and (b)$\Rightarrow$(c) are clear. To prove (c)$\Rightarrow$(a) assume that $h= nh$
for some $n >1$. Then $h(M)=n h(M)$, so $h(M)$ is either $0$ or $\infty$  for every $M\in\mathfrak M$. 
In other words, $h\in\HB(\mathfrak M)$.

\smallskip
Let $h^*= \sup \{nh: n\in \N_+\}$. Since $h$ is an entropy function, for every $n\in \N_+$ also $n h$ is an entropy function; this follows directly from the definition. 
Hence, $h^*$ is an entropy function as well. Since $h^b$ is binary, by item (b) we have $h^b=n h^b\succ nh$ for every $n\in\N_+$. Then $h^b\succ h^*$. 
If $M\in\mathfrak M$ and $h(M)=0$, then $nh(M)=0$ for every $n\in\N_+$ and so $h^*(M)=0$; if $h(M)>0$, then $h^*(M)\geq n h(M)$ for every $n\in\N_+$ and consequently $h^*(M)=\infty$. This shows that $h^*$ is binary. Clearly, $h^*\succ h$, and so $h^*\succ h^b$ by definition of binary hull, hence $h^*=h^b$.
As a by product we have seen that \eqref{h*} holds.
\end{proof}

\begin{remark}\label{NewRemark}
If $h, h_1:\mathfrak M\to \R_+\cup\{\infty\}$ are two  functions such that $h(M)= 0$ if and only if $h_1(M)=0$ for every $M \in \mathfrak M$, then 
$h\in\H(\mathfrak M)$ if and only if $h_1\in\H(\mathfrak M)$. In other words, the only value that matters is $0$, when we have to test when a function $h:\mathfrak M\to \R_+\cup\{\infty\}$ is an entropy function.
\end{remark}

\subsection{The Pinsker radical}\label{p-sec}

A \emph{preradical} $\r$ of $\mathfrak M$ is a subfunctor $\r:\mathfrak M\to\mathfrak M$ of the identity functor $id_\mathfrak M:\mathfrak M\to \mathfrak M$; equivalently, for all $M\in\mathfrak M$ there is a monomorphism $\r(M)\to M$ so that every morphism $f: M \to N$ in $\mathfrak M$ restricts to $\r(f):\r(M)\to \r(N)$ (i.e., $f(\r(M))\subseteq \r(N)$).

The preradical $\r$ is:
\begin{itemize}
  \item[(a)] a \emph{radical}, if $\r(M/\r (M))=0$ for every $M\in\mathfrak M$;
  \item[(b)] \emph{idempotent}, if $\r(\r (M))=\r(M)$ for every $M\in\mathfrak M$;
  \item[(c)] \emph{hereditary}, if $\r (N)=N \cap \r (M)$ for every $M\in\mathfrak M$ and every subobject $N$ of $M$. 
\end{itemize}
Hereditary preradicals are idempotent, but need not be radicals. A radical need not be idempotent. 

\begin{definition}
Let $h\in\H(\mathfrak M)$. The \emph{Pinsker radical} $\P_h:\mathfrak M\to\mathfrak M$ with respect to $h$ is defined for every object $M$ of $\mathfrak M$ as the join $$\P_h(M)=\sum\{N_j\subseteq M: h(N_j)=0\}.$$
\end{definition}

If $\mathfrak M=\mod_A$ for some ring $A$, then for every $M\in\mathfrak M$, $\P_h(M)=\sum\{Am:m\in M\ \text{and}\ h(Am)=0\}$.

\begin{lemma}\label{hP=0}
Let $h\in\H(\mathfrak M)$ and $M\in\mathfrak M$. Then $h(\P_h(M))=0$ and $\P_h(M)$ is the greatest subobject of $M$ with this property.
\end{lemma}
\begin{proof} 
Clearly, $\P_h(M)$ is a subobject of $M$, and if $N$ is a subobject of $M$ with $h(N)=0$, then $N\subseteq\P_h(M)$.
By definition, $\P_h(M)$ is the join of the family $\{N_j\subseteq M: h(N_j)=0\}$. In view of (A3) and Discussion \ref{A3}(c), we have $h(\P_h(M))=0$.
\end{proof}

\begin{theorem}\label{TT}
Let $h\in\H(\mathfrak M)$. Then $\P_h:\mathfrak M\to\mathfrak M$ is a hereditary radical.
\end{theorem}
\begin{proof}
First we see that $\P_h:\mathfrak M\to\mathfrak M$ is a functor. Indeed, $\P_h(M)$ is a subobject of $M$ for every $M\in\mathfrak M$. Moreover, it is well-defined on morphisms $f:M\to N$ in $\mathfrak M$. In fact, by the standard properties of images and joins, $$f(\P_h(M))=\sum\{f(N_j):N_j\subseteq M\ \text{and}\ h(N_j)=0\};$$ moreover, $h(N_j)=0$ implies $h(f(N_j))=0$ by (A2), and so $f(\P_h(M))\subseteq \P_h(N)$. Finally, if $g:N\to L$ is another morphism in $\mathfrak M$, then $\P_h(g\circ f)=\P_h(g)\circ\P_h(f)$ and $\P_h(id_M)=id_{\P_h(M)}$.

We prove now that $\P_h$ is a radical. Let $M\in\mathfrak M$.
Assume that $N$ is a subobject of $M/\P_h(M)$ such that $h(N)=0$. Let $N'$ be the pullback of $N$ along the projection $M\to M/\P_h(M)$. By (A2) applied to $N'$, $\P_h(M)$ and $N'/\P_h(M)=N$, we can conclude that $h(N')=0$. By the definition of $\P_h(M)$, this yields $N'\subseteq \P_h(M)$, and so $N=0$. Consequently $\P_h(M/\P_h(M))=0$ and hence $\P_h$ is a radical.

To show that the radical $\P_h$ is hereditary, consider $M\in\mathfrak M$ and a $N\subseteq M$. It is clear that $\P_h(N)\subseteq N\cap \P_h(M)$. By Lemma \ref{hP=0} we have $h(\P_h(M))=0$, and so $h(N\cap \P_h(M))=0$ by (A2). Since $N\cap \P_h(M)\subseteq N$, Lemma \ref{hP=0} implies $N\cap\P_h(M)\subseteq \P_h(N)$.
\end{proof}

The equivalent definition of the binary hull given by Proposition \ref{bheq} has an easy consequence on the Pinsker radical:

\begin{corollary}\label{Ph=Phb}
If $h\in\H(\mathfrak M)$, then $\P_h=\P_{h^b}$.
\end{corollary}


\subsection{The torsion theory associated to an entropy function}\label{tt-sec}

For the following definition and discussion we refer to \cite{D} and \cite{Gol}.

\begin{definition}
A pair $(\mathcal T,\mathcal F)$ of non-empty subclasses of an abelian category $\mathfrak M$ is a \emph{torsion theory} in $\mathfrak M$ if:
\begin{itemize}
\item[(i)] $\mathcal T\cap \mathcal F=\{0\}$;
\item[(ii)] if $T\to A\to 0$ is exact with $T\in\mathcal T$, then $A\in\mathcal T$; 
\item[(iii)] if $0\to A\to F$ is exact with $F\in\mathcal F$, then $A\in\mathcal F$; 
\item[(iv)] for every $M\in\mathfrak M$ there exists an exact sequence $0\to T\to M\to F\to 0$ with $T\in\mathcal T$ and $F\in\mathcal F$. 
\end{itemize}
\end{definition}

The class $\mathcal T$ is the class of \emph{torsion} objects in $\mathfrak M$, while $\mathcal F$ is the class of \emph{torsion-free} objects in $\mathfrak M$.

\begin{remark}\label{rremark}
A pair $(\mathcal T,\mathcal F)$ of non-empty subclasses of $\mathfrak M$ is a \emph{torsion theory} in $\mathfrak M$ if and only if:
\begin{itemize}
\item[(a)] $\Hom(T,F)=0$ for every $T\in\mathcal T$ and every $F\in\mathcal F$; and
\item[(b)] $\mathcal T$ and $\mathcal F$ are maximal with respect to (a).
\end{itemize}
Moreover, by \cite[Theorem 2.3]{D},
\begin{itemize}
\item[(i)] $\mathcal T$ is a torsion class if and only if it is closed under quotients, coproducts and extensions in $\mathfrak M$; and dually,
\item[(ii)] $\mathcal F$ is a torsion-free class if and only if it is closed under subobjects, products and extensions in $\mathfrak M$.
\end{itemize}
\end{remark}

For two torsion theories $\mathfrak t_j= ({\mathcal T}_j, {\mathcal F}_j)$, $j=1,2$, in $\mathfrak M$ one sets ${\mathfrak t}_1 \leq {\mathfrak t}_2$ if and only if $\mathcal F_1 \subseteq \mathcal F_2$ (or, equivalently, $\mathcal T_1 \supseteq\mathcal T_2$).

\begin{definition}
A torsion theory $(\mathcal T,\mathcal F)$ in $\mathfrak M$ is \emph{hereditary} if $\mathcal T$ is closed under subobjects.
\end{definition}

If $\r$ is an idempotent preradical of $\mathfrak M$, let
$$
\mathcal T_\r=\{M\in\mathfrak M: \r(M)=M\}\ \text{and}\ \mathcal F_\r=\{M\in\mathfrak M:\r(M)=0\}.
$$    
It is known that if $\r$ is a (hereditary) radical of $\mathfrak M$, then $\mathfrak t_\r=(\mathcal T_\r, \mathcal F_\r)$ is a (hereditary) torsion theory.

\smallskip
Every torsion theory $\mathfrak t=(\mathcal T,\mathcal F)$ in $\mathfrak M$ can be obtained by means of a radical $\r$.
Indeed, by \cite[Proposition 2.4]{D}, for every object $M$ in $\mathfrak M$ there exists a unique greatest subobject $M_\mathfrak t$ of $M$ such that $M_\mathfrak t\in\mathcal T$ and $M/M_\mathfrak t\in\mathcal F$ (with $M_\mathfrak t=\sum\{T\subseteq M:T\in\mathcal T\}$).
Moreover, by \cite[Corollary 2.5]{D} the correspondence $\r_\mathfrak t:M\mapsto M_\mathfrak t$ is an idempotent radical of $\mathfrak M$ (if $\mathfrak t$ is hereditary, then $\r_\mathfrak t$ is hereditary too). So starting from an idempotent radical $\r$ of $\mathfrak M$, we have that $\r_{\mathfrak t_\r}=\r$. In other words, there is a bijective correspondence between of all  idempotent radicals of $\mathfrak M$ and all torsion theories in $\mathfrak M$, so that hereditary radicals of $\mathfrak M$ correspond to hereditary torsion theories in $\mathfrak M$.

\begin{definition}\label{tr-closure}
For a torsion theory $\mathfrak t_\mathbf r$ in $\mathfrak M$, a subobject $N$ of $M\in\mathfrak M$ is $\mathfrak t_\mathbf r$-{\em closed} if $\mathbf r(M/N) =0$. 

More generally, the $\mathfrak t_\mathbf r$-{\em closure}  $\mathrm{cl}_{\mathfrak t_\mathbf r}(N)$ of a subobject $N$ of $M$ is the pullback of $\mathbf r(M/N)$ along the projection $M\to M/N$.
\end{definition}

Note that the assignment $N \mapsto \mathrm{cl}_{\mathfrak t_\mathbf r}(N)$ is a closure operator in the sense of \cite{DGiu,DT}. 

For example, if $\mathfrak M=\abg$ and $\mathbf r(G)= t(G)$, where $G$ is a torsion-free abelian group, this gives the classical notions of pure subgroup and purification, respectively.  

\medskip
For an entropy function $h$ of $\mathfrak M$, the two classes $\mathcal T_h$ and $\mathcal F_h$ defined in the introduction can be expressed in terms of the Pinsker radical as $$\mathcal T_h=\{M\in\mathfrak M: \P_h(M)=M\}\ \text{and}\ \mathcal F_h=\{M\in\mathfrak M:\P_h(M)=0\}.$$
In other words, $\mathcal T_h=\mathcal T_{\P_h}$ and $\mathcal F_h=\mathcal F_{\P_h}$. We abbreviate $\mathfrak t_{\P_h}$ simply to $\mathfrak t_h$.
According to Theorem \ref{TT}, $\mathfrak t_h$ is a hereditary torsion theory in $\mathfrak M$, so we give the following

\begin{definition}
For $h\in\H(\mathfrak M)$ we call $\mathfrak t_h$ the \emph{Pinsker torsion theory} of $h$ in $\mathfrak M$.
\end{definition}

Since for $M\in\mathfrak M$, $h(M)=0$ if and only if $\P_h(M)=M$, Remark \ref{rremark}(i) and Definition \ref{h-def} implies that $\mathcal T_h$ is the torsion class of a hereditary torsion theory.

If $h(M)>0$ for every non-zero $M\in\mathfrak M$, then $\P_h(M)=0$. In this case $\mathcal T_h=\{0\}$ and $\mathcal F_h=\mathfrak M$. 
 
\begin{theorem}\label{tt}\label{hb<->tt}
If $h\in\H(\mathfrak M)$, then $\mathfrak t_h$ is a hereditary torsion theory in $\mathfrak M$. The assignment $h \mapsto {\mathfrak t}_h = (\mathcal T_h, \mathcal F_h)$
determines a bijective order preserving correspondence between binary entropy functions of $\mathfrak M$ and hereditary torsion theories ${\mathfrak t}= ({\mathcal T}, {\mathcal F})$ in $\mathfrak M$. 
\end{theorem}
\begin{proof}
In view of Corollary \ref{Ph=Phb}, $\mathfrak t_h$ coincides with the torsion theory $\mathfrak t_{h^b}$ generated by the binary hull $h^b$ of $h$.
So it remains to define, for a given hereditary torsion theory ${\mathfrak t}= ({\mathcal T}, {\mathcal F})$, a binary entropy function $h$ of $\mathfrak M$ such that $\mathfrak t=\mathfrak t_h$. 
For $M\in\mathfrak M$ let $$h(M)=\begin{cases}0 & \text{if and only if}\ M\in\mathcal T,\ \text{and}\\ \infty & \text{otherwise}.\end{cases}$$ We check that $h$ is a
binary entropy function and that $\mathfrak t=\mathfrak t_h$.
The axiom (A1) is satisfied by $h$ since $\mathcal T$ and $\mathcal F$ are stable under isomorphisms. Consider now (A2); if for some $M\in\mathfrak M$, $h(M)=0$ and $N \subseteq M$, then $M\in\mathcal T$, which is closed under quotients and subobjects, so that $h(N)=0=h(M/N)$. Viceversa, if $h(N)=0=h(M/N)$ for some $M\in\mathfrak M$ and $N\subseteq M$, then $N\in\mathcal T$ and $M/N\in\mathcal T$; since $\mathcal T$ is closed under extensions, also $M\in\mathcal T$, that is, $h(M)=0$. Finally (A3) holds for $h$, since $\mathcal T$ is closed under coproducts. Then $h$ is a binary entropy function of $\mathfrak M$. By the definition of $h$ we have immediately $\mathcal T=\mathcal T_h$. It remains to show that $\mathcal F=\mathcal F_h$. For every $M\in\mathfrak M$ there exists $T\in\mathcal T$ such that $M/T\in\mathcal F$; moreover, $M\in\mathcal F$ if and only if $T=0$. This occurs if and only if $h(N)=\infty$ for every non-zero subobject $N$ of $M$, that is, $h(M)>\!\!>0$, i.e., $M\in\mathcal F_h$. Hence $\mathcal F=\mathcal F_h$ and we have proved that $\mathfrak t=\mathfrak t_h$.
Finally, it is clear that the correspondence preserves the order.
\end{proof}

This theorem shows that hereditary torsion theories in $\mathfrak M$ are nothing else but binary entropy functions of $\mathfrak M$.

\begin{remark}\label{minimal} 
The set of axioms (A1), (A2), (A3) used in Definition \ref{h-def} to define an entropy function $h:\mathfrak M\to \R_+\cup\{\infty\}$ is minimal in order to prove Theorem \ref{tt}. Indeed, first of all it is natural to impose (A1) since the classes $\mathcal T_h$ and $\mathcal F_h$ are stable under isomorphisms. Moreover, we want $\mathcal T_h$ to be a torsion class of a hereditary torsion theory. Since $\mathcal T_h$ has to be closed under subobjects, quotients and extensions, hence we have to impose (A2). Finally, $\mathcal T_h$ has to be closed under coproducts, that is, $h$ satisfies (A3).
\end{remark}


\section{Entropy functions of algebraic flows}\label{ef-af-sec}

\subsection{The category of flows}\label{flow-sec}

In this section we recall the category of flows for a given arbitrary category $\mathfrak X$.

\smallskip
We introduce the category $\af_{\mathfrak X}$ of flows of $\mathfrak X$ first as a subcategory of the arrows category $\mathfrak X^2$.
Recall that for a category $\mathfrak X$ the category of arrows $\mathfrak X^2$ is isomorphic to a special \emph{comma category}, namely $(\mathfrak X\downarrow\mathfrak X)$. The objects of this category $\mathfrak X^2$ are triples $(X,Y,f)$, where $X,Y$ are objects of $X$ and $f: X \to Y$ is a morphism in ${\mathfrak X}$. The morphisms between two objects $(X_1,Y_1,f_1)$ and $(X_2,Y_2,f_2)$ of $\mathfrak X^2$ are pairs $(u,v)$ of morphisms $u: X_1\to X_2$ and $v: Y_1 \to Y_2$ in $\mathfrak X$ such that the diagram
\begin{equation}\label{diagram-1}
\begin{CD}
X_1 @>f_1>> Y_1\\ 
@V u VV  @VV v V\\ 
X_2 @>>f_2> Y_2
\end{CD}
\end{equation}
commutes. 

To describe $\af_{\mathfrak X}$ we impose now two restrictions. First, 
we consider special objects of the arrows category $\mathfrak X^2$ of $\mathfrak X$, namely, we take endomorphisms in $\mathfrak X$ instead of all morphisms:  

\begin{definition}
A \emph{flow} in ${\mathfrak X}$ is a pair $(X,\phi)$, where $X$ is an object in $\mathfrak X$ and $\phi: X \to X$ an endomorphism in ${\mathfrak X}$. 
\end{definition} 

So the category $\af_{\mathfrak X}$ has as objects all flows in ${\mathfrak X}$. 
Second, $\af_{\mathfrak X}$ will not be a full subcategory of $\mathfrak X^2$, since we shall take as morphisms in $\af_{\mathfrak X}$ only those pairs $(u,v)$ in \eqref{diagram-1} with $u=v$. Namely, a morphism in $\af_{\mathfrak X}$ between two flows $(X,\phi)$ and $(Y,\psi)$ is a morphism $u: X \to Y$ in ${\mathfrak X}$ such that the diagram
\begin{equation}\label{casc-mor}
\begin{CD}
X @>\phi>> X\\ 
@V u VV  @VV u V\\ 
Y @>>\psi> Y
\end{CD}
\end{equation}
in ${\mathfrak X}$ commutes. Two flows $(X,\phi)$ and $(Y,\psi)$ are isomorphic in $\af_{\mathfrak X}$ if the morphism $u: X \to Y$ in \eqref{casc-mor}
is an isomorphism in ${\mathfrak X}$. 

\medskip
Actually, $\af_\mathfrak X$ is isomorphic to a functor category. Indeed, consider the monoid $\N$ as a one-object category and the functor category $\mathbf{Fun}(\N,\mathfrak X)$. Since $\N$ has one object $Z$ and the morphisms of this object are the free monoid $\N$ generated by $1$, every $F\in\mathbf{Fun}(\N,\mathfrak X)$ is determined by $F(Z)$ and $F(1)$; and if $F,G\in\mathbf{Fun}(\N,\mathfrak X)$, a morphism from $F$ to $G$ is given by the natural transformation $\{u\}$, where $u:F(Z)\to G(Z)$ is a morphism such that $u\circ F(1)=G(1)\circ u$.

\medskip
At this point the functor $I:\af_\mathfrak X\to \mathbf{Fun}(\N,\mathfrak X)$, which associates to a pair $(X,\phi)\in\af_\mathfrak X$ the functor $F_{(X,\phi)}:\N\to \mathfrak X$ such that $F_{(X,\phi)}(Z)=X$ and $F_{(X,\phi)}(1)=\phi$ and to a morphism $u:(X,\phi)\to (Y,\psi)$ the natural transformation $\{u\}$ from $F_{(X,\phi)}$ to $F_{(Y,\psi)}$, is an isomorphism of categories.

\smallskip
It is known that, the category $\af_\mathfrak X$ is abelian, if $\mathfrak X$ is abelian.

\begin{remark}
When $\mathfrak X$ is a concrete category, every flow $(X,\phi)$ in $\mathfrak X$ gives a  semigroup action of $\N\cong \{\phi^n:n\in \N\}$ on $X$ (and viceversa, a semigroup action $\alpha: \N\times X\to X$ of $\N$ on $X$ via endomorphisms of $X$ defines a flow $(X,\alpha(1,-))$). In case $\phi: X \to X$ is an automorphism, this action becomes a group action of $\Z\cong  \{\phi^n:n\in \Z\}$. 
\end{remark}

\subsection{Algebraic flows and module categories}

Let $R$ be a ring. As said in the introduction we abbreviate $\af_{\mod_R}$ simply to $\af_R $, and we call algebraic flow an element $(M,\phi)$ of $\af_R$. 

An algebraic flow $(M,\phi)$ in $\af_R$ can be completely ``encoded'' via a structure of a $R[t]$-module of $M$. This gives the following theorem, that in particular shows that $\af_R$ is an abelian category.
 
\begin{theorem}\label{iso}
Let $R$ be a ring. Then $\af_R$ and $\mod_{R[t]}$ are isomorphic categories.
\end{theorem}
\begin{proof}
Let 
\begin{equation}\label{F}
F:\af_R\to\bim
\end{equation}
 be defined in the following way.
If $(M,\phi)\in\af_R$, define an $R[t]$-module structure on $M$ by letting $tm=\phi(m)$ for every $m\in M$. This can be extended to multiplication by arbitrary polynomials in obvious way. We denote $M$ with this structure by $M_\phi$. Let $F(M,\phi)=M_\phi$. If $u:(M,\phi)\to (N,\psi)$ is a morphism in $\af_R$, let $F(u)=u:M_\phi\to N_\psi$. Then $u: M_\phi\to N_\psi$ is a morphism of $R[t]$-modules. 
It is easy to see that $F$ is a functor.

Viceversa, let 
\begin{equation}
G:\bim\to \af_R
\end{equation}
 be defined in the following way. Every $R[t]$-module $M$ gives rise to an algebraic flow $(M,\mu_M)$, where $\mu_M(m)=t m$ for every $m\in M$. So let $G(M)=(M,\mu_M)$. Moreover, if $u:M\to N$ is a morphism in $\bim$, let $G(u)=u:(M,\mu_M)\to (N,\mu_N)$. This is a morphism in $\af_R$, since $u\circ\mu_M=\mu_N\circ u$. 
It is easy to see that $G$ is a functor and that $F$ and $G$ give an isomorphism between $\af_R$ and $\bim$.
\end{proof}

This isomorphism is very convenient since it allows us to replace $\af_R$ by a module category, namely, $\mod_{R[t]}$. 
 
\medskip
Theorem \ref{hb<->tt} shows that there are at least as many entropy functions as hereditary torsion theories. Since our aim is to concentrate mainly on those entropy functions having a ``reasonable behavior'' from a dynamical point of view, we shall impose some further restrictions in the form of axioms in addition to those of the general Definition \ref{h-def}. This is the aim of the next section, in which we concentrate on the particular case $\mathfrak M=\mod_A$, where $A$ is ring.

\subsection{Adding some axioms}\label{ax-sec}

The known entropy functions $h\in \mathfrak H(\mod_A)$ satisfy a stronger form of (A2), namely,

\begin{itemize}
\item[(A2$^*$)] $h(M)=h(N)+h(M/N)$, for every $M\in\mod_A$ and every submodule $N$ of $M$. 
\end{itemize}

Following the standard terminology in module theory, we call an entropy function satisfying (A2$^*$) {\em additive}. 
In case $A$ is an integral domain, (A2$^*$) has the following consequences.

\begin{lemma}\label{hI=hA}
Let $A$ be an integral domain and $h\in\H(\mod_A)$. Then $h(I)=h(A)$ for every non-zero ideal of $A$.
\end{lemma}
\begin{proof}
Pick any non-zero $a\in I$. Since $Aa\cong A$ and $Aa\subseteq I$, we have the inequalities $h(A) = h(aA) \leq h(I) \leq h(A)$, hence $h(I)=h(A)$.
\end{proof}

\begin{proposition}\label{i(R)_fin} 
Let $A$ be an integral domain and $h\in \H(\mod_A)$. If $h$ is additive and $0 <h(A) < \infty$, then $h(t(M)) = 0$; consequently $h(M) = h(M/t(M))$ for every $M \in \mod_A$.  
\end{proposition}
\begin{proof}
Let $r= h(A)$ and let $I$ be a non-zero ideal of $A$. Then $h(I)=r$ by Lemma \ref{hI=hA}. 
From (A2$^*$) applied to $A$ and $I$ we deduce $r= r + h(A/I)$, so $h(A/I)=0$. 
Now take any $M \in \mod_A$.  If $x\in t(M)$, then $I = \mathrm{ann}_A(x)$ is a non-zero ideal of $A$, therefore $Ax \cong A/I$ and consequently $h(Ax) = h(A/I)=  0$. This easily gives  $h(t(M)) = 0$. Another application of  (A2$^*$) yields $h(M) = h(M/t(M))$.
\end{proof}

Let $A$ be an integral domain and let $h\in \H(\mod_A)$ be additive with $h(A) = \infty$. Call an ideal $I$ of $A$ {\em $h$-large}, if $h(A/I)< \infty$, denote by 
 $\mathcal J_h(A)$ the family  of all {\em $h$-large} ideals of $A$. Let now $r_I = h(A/I)$ for  $I\in \mathcal J_h(A)$. 
 
\begin{proposition}\label{Impose}
Let $A$ be an integral domain and $h\in \H(\mod_A)$. If $A$ is additive and $h(A)= \infty$, then:
\begin{itemize}
\item[(a)] $\mathcal J_h(A)$ is closed under taking finite intersections and larger ideals; therefore, $\mathcal J_h(A)$ is the local base of a ring topology $\tau_h$ on $A$;
\item[(b)] $\tau_h$ is indiscrete $($i.e., $\mathcal J_h(A)$ coincides with $\{A\})$ exactly when $h(M) =\infty$ for every non-zero $A$-module $M$; 
\item[(c)] $r_I = h(J/I) + r_J$ when $I\subseteq J$ in $\mathcal J_h(A)$;
\item[(d)] $r_{I\cap J} = r_I + r_J$ when $I,J\in \mathcal J_h(A)$ are coprime (i.e., $I+ J =A$);
\item[(e)] if $A$ is a PID, then  $r_{Abc}=r_{Ab}+r_{Ac}$ for $Ab ,Ac\in \mathcal J_h(A)$.  
\end{itemize}
 \end{proposition}
\begin{proof} 
(a) Follows from $A/I \cap J \hookrightarrow A/I \times A/J$ for  $I,J\in \mathcal J_h(A)$ and the monotonicity of $h$ under taking quotients.

\smallskip
(b) Is obvious, and (c) and (d) follows from the additivity of $h$ and the isomorphism $A/I \cap J \cong A/I \times A/J$ for coprime ideals $I,J$ of $A$. 

\smallskip
(e) If  $I\subseteq J$ are ideals in $\mathcal J_h(A)$, then they have the form $J=Ab$ and $I = Jc = Abc$, so $J/I \cong A/Ac$. This gives $r_{Abc}=r_{Ab}+r_{Ac}$. 
\end{proof}

It follows from item (e) that the numbers $r_{Ap}$ for prime elements $p\in A$ determine all others. 

\medskip
In the sequel we consider only entropy functions $h$ of $\af_R$ (i.e., $A = {R[t]}$). Because of the importance of the additivity axiom (A2$^*$), we reformulate it in this specific case: 
 
\begin{itemize}
\item[(A2$^*$)] $h(\phi)=h(\phi\restriction_N)+h(\overline\phi)$, for every $(M,\phi)\in\af_R$ and every $\phi$-invariant submodule $N$ of $M$, with $\overline\phi:M/N\to M/N$ the endomorphism induced by $\phi$.
\end{itemize}

In this form it is known also as {\em Addition Theorem} --- see Theorem \ref{AT} for the algebraic entropy $h_a$, Fact \ref{properties-ent}(e) for the algebraic entropy $\ent$ and Theorem \ref{AT-i} for the $i$-entropy.

\medskip
A starting axiom that permits to avoid the pathological case of $h\equiv\infty$, is the following one, which intuitively has to hold for a reasonable entropy function $h\in\H(\af_R)$:
\begin{itemize}
\item[(A0)] $h(0_M)=0$ and $h(1_M)=0$ for every $M\in\mod_R$.
\end{itemize}
We split (A0) in two parts in order to better deal with it:
\begin{itemize}
\item[(A0$_0$)] $h(0_M)=0$ for every $M\in\mod_R$;
\item[(A0$_1$)] $h(1_M)=0$ for every $M\in\mod_R$.
\end{itemize}

Observe that  in the presence of (A2) and (A3),  the axioms (A0$_0$) and (A0$_1$) follow respectively from the equalities $h(0_R)=0$ and $h(1_R)=0$, since every $R$-module is quotient of a free $R$-module.

\medskip
Compositions of endomorphisms are available in $\af_R$, so for every algebraic flow $(M,\phi) \in \af_R$ one can consider all powers $\f^n$ to get new algebraic flows $(M,\phi^n)\in\af_R$ with the same underlying $R$-module $M$. 
This gives the possibility to consider the following axiom for an entropy function $h$ of $\af_R$: 

\begin{itemize}
\item[(A4)] $h(\f)\leq h(\f^n) \leq nh(\f)$ for every $(M,\phi)\in\af_R$.
\end{itemize}
Most of the known specific examples of entropy functions satisfy this weak logarithmic law (see Fact \ref{properties}(b) for the algebraic entropy $h_a$, Fact \ref{properties-ent}(b) for the algebraic entropy $\ent$ and Fact \ref{properties-i}(b) for the $i$-entropy), or the (stronger) logarithmic law:
\begin{itemize}
\item[(A4$^*$)] $h(\f^n)=nh(\f)$ for every $(M,\phi)\in\af_R$.
\end{itemize}

\begin{remark}
Let $h\in\H(\af_R)$. In other terms (A4$^*$) says that, fixed $M\in\mod_R$, the restriction $h:(\End(M),\cdot)\to(\R_+\cup\{\infty\},+)$ is a semigroup homomorphism. Since every semigroup homomorphisms sends idempotents to idempotents, and the only idempotents of $(\R_+\cup\{\infty\},+)$ are $0$ and $\infty$, we have that idempotency of $\phi\in\End(M)$ implies that $h(\phi)$ is either $0$ or $\infty$.
\end{remark}

In particular we have the following obvious

\begin{lemma}\label{0-infty}
Let $h\in\mathfrak H(\af_R)$. If $h$ satisfies \emph{(A4$^*$)}, then $h(0_M)$ and $h(1_M)$ are either $0$ or $\infty$ for every $M\in\mod_R$.
\end{lemma}

Consider the functors 
\begin{itemize}
\item[(a)]$O:\mod_R\to \af_R\cong\mod_{R[t]}$ such that $M\mapsto (M,0_M)$, and 
\item[(b)] $I:\mod_R\to \af_R\cong\mod_{R[t]}$ such that $M\mapsto (M,1_M)$. 
\end{itemize}
Then both $O(\mod_R)$ and $I(\mod_R)$ are subcategories of $\af_R$ closed under quotients, direct sums and modules, and isomorphic to $\mod_R$. 

Let now $h\in\mathfrak H(\af_R)$. We can associate to this entropy function $h$ of $\af_R$, 
\begin{itemize}
\item[(a)] $h^O:\mod_R\to \R_+\cup\{\infty\}$ defined by $h^O(M)=h(0_M)\ \text{for every}\ M\in\mod_R$; and
\item[(b)]$h^I:\mod_R\to \R_+\cup\{\infty\}$ defined by $h^I(M)=h(1_M)\ \text{for every}\ M\in\mod_R$. 
\end{itemize}

\begin{remark}
Let $h\in\mathfrak H(\af_R)$.
\begin{itemize}
\item[(a)]If $h$ satisfies (A4$^*$), then $h^O\in\mathfrak H_b(\mod_R)$ and $h^I\in\mathfrak H_b(\mod_R)$.

In fact, $h^O$ and $h^I$ are entropy functions of $\mod_R$, since $h^O=h\restriction_{O(\mod_R)}$ and $h^I=h\restriction_{I(\mod_R)}$, where $O(\mod_R)$ and $I(\mod_R)$ in $\af_R$ are closed under quotients, direct sums and submodules, and are isomorphic to $\mod_R$. Moreover, in view of Lemma \ref{0-infty}, (A4$^*$) implies that $h(0_M)$ and $h(1_M)$ are either $0$ or $\infty$, and so $h^O$ and $h^I$ are binary. 

\item[(b)] Consequently, if $h$ satisfies (A4$^*$), then:
\begin{itemize}
\item[(b$_1$)]$h$ satisfies (A0$_0$) if and only if $h^O\equiv 0$;
\item[(b$_2$)]$h$ satisfies (A0$_1$) if and only if $h^I\equiv 0$.
\end{itemize}
In other words, $h^O$ and $h^I$ measure to what extent $h$ satisfies (A0$_0$) and (A0$_1$), respectively.
\end{itemize}
\end{remark}

In the next section we will see that (A0$_0$) does not imply (A0$_1$) and that (A0$_1$) does not imply (A0$_0$).

\begin{remark}
The abstractly defined binary hull $h^b$ of an entropy function $h$ of $\mathfrak M$ has a very natural intuitive construction in the case of algebraic flows. Indeed, from a given entropy function $h$ of $\af_R$, for every $n\in\N_+$  one can define a new entropy function $h_n$ of $\af_R$ letting
$h_n(M,\phi)= h(M,\phi^n)$ for every $(M,\phi) \in \af_R$ (one can define this also in the isomorphic situation $M_\phi\in\mod_{R[t]}$). 
\begin{itemize}
\item[(a)] If $h\in\HB(\af_R)$, then $h=nh$ for every $n\in\N_+$.

Indeed, according to Proposition \ref{bheq}, $h = nh$ for every $n\in\N_+$, since $h$ is binary.

\item[(b)] If $h\in\HB(\af_R)$ satisfies {(A4)}, then $h=h_n=nh$ for every $n\in\N_+$.

In fact, $h \prec h_n \prec nh\ \text{for all}\ n\in\N_+$, and so $h=h_n=nh$ for all $n\in\N_+$.
\item[(c)]
If $h\in\H(\af_R)$ satisfies {(A4)}, then $h^b=\sup_{n\in\N_+}h_n$.

To verify this statement, let $h'=\sup_{n\in\N_+}h_n$. Then $h'\in\HB(\af_R)$. Moreover, suppose that $h^*\in\HB(\af_R)$ and $h\prec h^*$. By item (b) we have $h_n \prec h_n^*= h^*$ for every $n\in\N_+$, so that also $h' \prec h^*$. This proves that $h'$ is the least binary entropy function above $h$, so $h'$ coincides with the binary hull $h^b$ of $h$.
\end{itemize}
\end{remark}

For $M\in\mod_R$, the \emph{right Bernoulli shift} is the algebraic flow
$$\beta_M:M^{(\N)}\to M^{(\N)}\ \text{defined by}\ (x_0,\ldots,x_n,\ldots)\mapsto(0,x_0,\ldots,x_n,\ldots).$$

\begin{example}\label{betaex}
\begin{itemize}
\item[(a)] The algebraic flow $(R^{(\N)},\beta_R)\in\af_R$ is mapped by the functor $F:\af_R\to\mod_{R[t]}$ of \eqref{F} to $R[t]$. 
\item[(b)] Let $I$ be an ideal of $R$. 
As in (a), $F$ sends $((R/I)^{(\N)}, \beta_{R/I})$ to $(R/I)[t]$. 
Observe that $(R/I)[t]$ is isomorphic to the $R[t]$-module $R[t]/J$, where $J=I[t]$ is the ideal generated by $I$ in $R[t]$. 
\end{itemize}
\end{example}

The following condition appears in the collections of axioms that guarantee uniqueness for the algebraic entropy $h_a$ of $\af_\abg$ and for the algebraic entropy $\ent$ of $\af_\mathbf{TorAbGrp}$: $h_a(\beta_{\Z(p)})=\ent(\beta_{\Z(p)})=\log p$ for every prime $p$ (see Theorem \ref{UT} and Fact \ref{properties-ent}(f) below, respectively). 

\medskip
Inspired by this condition, in order to reach uniqueness of an arbitrary entropy function $h\in\H(\af_R)$, one may want to add the following axiom:
\begin{itemize}
\item[(A5)] $h(\beta_{R/I})=r_{I}\in\R_+\cup\{\infty\}$, for every ideal $I$ of $R$,
\end{itemize}
with appropriate conditions on the $r_{I}$ (e.g., imposed by Proposition \ref{Impose} when $R$ is an integral domain).

For example, if $h$ is monotone under taking quotients and $I,J$ are ideals of $R$, then $r_I\geq r_J$ if $I\subseteq J$. Moreover,  if (A2$^*$) holds for $h$, then
$r_I = h(\beta_{J/I}) + r_J$ when $I\subseteq J$, while $r_{I\cap J} = r_I + r_J$ when $I,J$ are coprime, etc.
 
\medskip
Beyond the quotients of $R[t]$ with respect to ideals as in Example \ref{betaex}(b), one can take also principal ideals of $R[t]$ of the form $J = (f(t))$ for some $f(t)\in R[t]$. These ideals also provide data that may help to capture the entropy function, namely the values $y_f= h(R[t]/J)$. 
In particular, if the polynomial $f(t) = a_0 + a_1t + \ldots + a_{n-1}t^{n-1}+ t^n$ is monic, then $R[t]/J$ is isomorphic to $F(R^n,\phi)$ as $R[t]$-modules, where $\phi$ is the endomorphism of $R^n$ having as a matrix the companion matrix of $f(t)$.

\medskip
For a ring $A$ and $M\in\mod_A$, recall from Remark \ref{A3d} that $\mathcal F(M)$ denotes the family of all finitely generated submodules of $M$. An invariant $i:\mod_A\to \R_+\cup\{\infty\}$ is \emph{upper continuous} if $i(M)=\sup\{i(F):F\in\mathcal F(M)\}$ for every $M\in\mod_A$ \cite{NR,V}. As follows from Remark \ref{A3d}, this property implies (A3).

\begin{definition}\emph{\cite{NR}}\label{L-def}
An additive upper continuous invariant is called \emph{length function}.
\end{definition}

Since each length function obviously satisfies (A1), (A2) and (A3), it follows that length functions are entropy functions. It is not hard to see that a length function $h$ is completely determined by its values on cyclic modules $h(A/I)$. 
If one restricts to Noetherian commutative rings $A$ and the entropy function $h\in\mathfrak H(\mod_A)$ is a length function, to prove uniqueness of $h$ one can eventually use the following

\begin{fact}\label{V}\emph{\cite{V}}
If $A$ is a Noetherian commutative ring and $i,i'$ are length functions of $\mod_{A}$, then $i=i'$ if and only if $i(A/\mathfrak p)=i'(A/\mathfrak p)$ for every prime ideal $\mathfrak p$ of $A$.
\end{fact}

This fact \ref{V}, combined with the fact that $h_a\in\mathfrak H(\af_\abg)$ is a length function of $\mod_{\Z[t]}$, 
 is used in \cite{DG} to give a second proof of the Uniqueness Theorem for $h_a$ (see Theorem \ref{UT}).

\medskip
The following general problem remains open.

\begin{problem}
Find a family of axioms that give uniqueness for an abstractly defined entropy function $h$ of $\af_R$, where $R$ is an integral domain.
\end{problem}

\section{Radicals capturing the dynamics of module endomorphisms}\label{qp-sec}

\subsection{Radicals of $\af_R$}

Let $R$ be a ring. We recall that if $(M,\phi)\in\af_R$ and $F$ is a subset of $M$, for any positive integer $n$, the \emph{$n$-th $\phi$-trajectory} of $F$ with respect to $\phi$ is $$T_n(\phi,F)=F+\phi(F)+\ldots+\phi^{n-1}(F),$$ and the $\phi$-trajectory of $F$ is $$T(\phi,F)=\sum_{n\in\N}\phi^n(F).$$

\smallskip
In \cite{DG1} the Pinsker subgroup $\P_{h_a}(G,\phi)$ of an algebraic flow $(G,\phi)\in\af_\abg$ was characterized also in terms of the quasi-periodic points of $\phi$.

\begin{definition}\label{DefQP}
Let $(M,\phi)\in\af_R$. An element $x\in M$ is a \emph{quasi-periodic point} of $\phi$ if there exist $n>m$ in $\N$ such that $\phi^n(x)=\phi^m(x)$. Moreover, $\phi$ is \emph{pointwise quasi-periodic} if every $x\in M$ is a quasi-periodic point of $\phi$; and $\phi$ is \emph{quasi-periodic} if there exist $n>m$ in $\N$ such that $\phi^n=\phi^m$.
\end{definition}

\smallskip
We generalize the definition given in \cite{DG1} to every $(M,\phi)\in\af_R$ letting by induction:
\begin{itemize}
\item[(a)] $Q_0(M,\f)=0$, and for every $n\in\N$
\item[(b)] $Q_{n+1}(M,\f)=\{x\in M: (\exists n>m \ \text{in}\ \N)\ (\phi^n-\phi^m)(x)\in Q_n(M,\phi)\}$.
\end{itemize}
This defines an increasing chain of $\phi$-invariant submodules of $M$ $$Q_0(M,\f)\subseteq Q_1(M,\f)\subseteq \ldots\subseteq Q_n(M,\f)\subseteq\ldots,$$ where $Q_1(M,\phi)$ is the submodule of $M$ consisting of all quasi-periodic points of $\phi$. 
In particular,  $\phi$ is pointwise quasi-periodic if and only if $M=Q_1(M,\phi)$. 
Let
 $$
 \QQ(M,\f):=\bigcup_{n\in\N} Q_n(M,\f).
 $$
Then also $\QQ(M,\phi)$ is a $\phi$-invariant submodule of $M$.

\medskip
Imitating the definition of $\QQ$, for every $(M,\phi)\in\af_R$ define by induction:
\begin{itemize}
\item[(a)] $O_0(M,\f)=0$, and for every $n\in\N$
\item[(b)] $O_{n+1}(M,\f)=\{x\in M: \phi(x)\in O_n(M,\phi)\}$.
\end{itemize}
This defines an increasing chain of $\phi$-invariant submodules of $M$ $$O_0(M,\f)\subseteq O_1(M,\f)\subseteq \ldots\subseteq O_n(M,\f)\subseteq \ldots.$$
Note that $O_n(M,\phi)=\ker\phi^n$ for every $n\in\N$.
Let $$\mathfrak O(M,\f):=\bigcup_{n\in\N} O_n(M,\f).$$
Then also $\mathfrak O(M,\phi)$ is a $\phi$-invariant submodule of $M$.

\smallskip
Also for the identical endomorphism it is possible to define by induction:
\begin{itemize}
\item[(a)] $I_0(M,\f)=0$, and for every $n\in\N$
\item[(b)] $I_{n+1}(M,\f)=\{x\in M: (\phi-1_M)(x)\in I_n(M,\phi)\}$.
\end{itemize}
This defines an increasing chain of $\phi$-invariant submodules of $M$ $$I_0(M,\f)\subseteq I_1(M,\f)\subseteq \ldots\subseteq I_n(M,\f)\subseteq \ldots.$$
Note that $I_1(M,\phi)$ is the submodule of $M$ of fixed points of $\phi$. Let $$\mathfrak I(M,\f):=\bigcup_{n\in\N} I_n(M,\f).$$
Then also $\mathfrak I(M,\phi)$ is a $\phi$-invariant submodule of $M$.

\smallskip
Note that $O_n(M,\phi)\subseteq Q_n(M,\phi)$ and $I_n(M,\phi)\subseteq Q_n(M,\phi)$ for every $n\in\N$. So also $\mathfrak O(M,\phi)\subseteq\mathfrak Q(M,\phi)$ and $\mathfrak I(M,\phi)\subseteq \mathfrak Q(M,\phi)$. One can prove directly that $\QQ$, $\mathfrak O$ and $\mathfrak I$ are radicals of $\af_R$; Corollary \ref{uff} will prove it easily in $\mod_{R[t]}$ using the isomorphism given by Theorem \ref{iso}.

Moreover, $\mathfrak O$ (respectively, $\mathfrak I$) is the smallest radical of $\af_R$ containing all zero endomorphisms (respectively, all identities). So, $M= \mathfrak O(M)$ (respectively, $M=\mathfrak I(M)$)
if and only if for every $x\in M$ there exists $n\in \N_+$ such that $\phi^n(x) =0$ (respectively, $(\phi-id_M)^n(x) =0$). 

\medskip
The property in (a) of the next lemma is proved in \cite{DG1} for $\af_\abg$. It is possible to prove it in the same way for $\af_R$, and analogously the properties in (b) and (c).

\begin{lemma}\label{Q/Q}
Let $(M,\phi)\in\af_R$ and $n\in\N$. Then:
\begin{itemize}
\item[(a)] $Q_n(M,\phi)$ is a $\phi$-invariant submodule of $M$, and for the induced endomorphism $\overline \f_n$ of $M/  Q_n(M,\f)$, $$Q_{n+1}(M,\f)/ Q_n(M,\f)= Q_1(M/ Q_n(M,\f),\overline\f_n);$$
\item[(b)] $O_n(M,\phi)$ is a $\phi$-invariant submodule of $M$, and for the induced endomorphism $\overline \f_n$ of $M/  O_n(M,\f)$,  $$O_{n+1}(M,\f)/ O_n(M,\f)= O_1(M/ O_n(M,\f),\overline\f_n);$$
\item[(c)] $I_n(M,\phi)$ is a $\phi$-invariant submodule of $M$, and for the induced endomorphism $\overline \f_n$ of $M/  I_n(M,\f)$, $$I_{n+1}(M,\f)/ I_n(M,\f)= I_1(M/ I_n(M,\f),\overline\f_n).$$
\end{itemize}
\end{lemma}

 For $(G,\phi)\in\af_\mathbf{TorAbGrp}$, we shall see in \eqref{DG1} that $\QQ(G,\phi)=Q_1(M,\phi)$, but this is not the case in general, as the following example shows.

\begin{example}\label{BIG-matrix}
Let $G=\Z^{(\N)}=\bigoplus_{n=1}^\infty\hull{e_n}$, and let $\phi$ be the automorphism of $G$ given by the matrix 
$$\begin{pmatrix} 
1 & 1 & 1 & \ldots \\
0 & 1 & 1 & \ldots \\
0 & 0 & 1 & \ldots \\
\vdots & \vdots & \ddots & \ddots
\end{pmatrix}.$$
For every $n\in\N_+$ let $G_n=\hull{e_1,\ldots,e_n}$. Then $G_n=Q_n(G,\phi)$ and $\QQ(G,\f)=G$. In particular, we have the following strictly increasing chain
$$0=Q_0(G,\phi)\subset Q_1(G,\phi)\subset Q_2(G,\phi)\subset \ldots\subset Q_n(G,\phi)\subset \ldots\subset \QQ(G,\phi)=G.$$
\end{example}

 From the construction of $\QQ(M,\phi)$ for $(M,\phi)\in\af_R$, one can realize  that 
$$\QQ(M,\phi)=\left\{x\in M: \exists k\in\N, (n_j,m_j)\in\N^2,n_j>m_j,j=0,\ldots,k-1, \left(\prod_{j=0}^{k-1}(\phi^{n_j}-\phi^{m_j})\right)(x)=0 \right\}.$$
Motivated by this equality, one can define now analogously another submodule for every $(M,\phi)\in\af_R$, letting
$$
\mathfrak A(M,\phi)=\left\{x\in M:\exists k\in\N, a_j\in R, j=0,\ldots,k-1,\phi^k(x)=\sum_{j=0}^{k-1}a_j\phi^j(x) \right\},
$$
i.e., $x\in \mathfrak A(M,\phi)$ if and only if $x\in T_k(\phi,Rx)$ (or, equivalently, $T_k(\phi,Rx)$ is a $\phi$-invariant submodule) for some $k\in\N_+$. 
Then $\mathfrak A(M,\phi)$ is a $\phi$-invariant submodule of $M$ and 
 $\QQ(M,\phi)\subseteq \mathfrak A(M,\phi)$ for every $(M,\phi)\in\af_R$. We shall see in Corollary \ref{uff} that also $\mathfrak A$ is a radical.

\smallskip
According to \cite{DGSZ}, an endomorphism $\phi$ of $M\in\mod_R$ is \emph{pointwise integral} if for every $x\in M$ there exists a monic polynomial $p(t)\in R[t]$ such that $p(\phi)(x)=0$, in other words, if $M=\mathfrak A(M,\phi)$. 
According to \cite{Simone}, $\phi$ is \emph{pointwise algebraic} if for every $x\in M$ there exists a polynomial $p(t)\in R[t]\setminus\{0\}$ such that $p(\phi)(x)=0$. Clearly, pointwise integral implies pointwise algebraic. So, for $(M,\phi)\in\af_R$ we introduce also $$\mathfrak W(M,\phi)=\{x\in M:\exists p(t)\in R[t]\setminus\{0\},p(\phi)(x)=0\},$$
and obviously $\mathfrak A(M,\phi)\subseteq \mathfrak W(M,\phi)$. Clearly, $\phi$ is pointwise algebraic if and only if $M=\mathfrak W(M,\phi)$.
If $R$ is an integral domain, then $\mathfrak W(M,\phi)$ is a $\phi$-invariant submodule of $M$, and so $\mathfrak W$ is another radical of $\af_R$, so that 
$$
\QQ(M,\phi)\subseteq\mathfrak A(M,\phi)\subseteq \mathfrak W(M,\phi)
$$
for every $(M,\phi)\in\af_R$.
\begin{example}\label{LAST-ex}
Let $R=\Z$. If $G$ is a torsion abelian group, then
$$
Q_1(G,\phi)= \mathfrak Q (G, \phi)= \mathfrak A(G,\phi) \subseteq \mathfrak W(G,\phi)=G\mbox{ for every }\phi \in \End(G).
$$
\begin{itemize}
\item[(a)] If $p$ is a prime number and $G =\Z(p)^{(\N)}$, then $\mathfrak W(G,\phi)=G$ 
for every $\phi\in \End(G)$, while $\mathfrak A(G,\beta_{\Z(p)})=0$. 
\item[(b)] If $G$ is a torsion-free abelian group of finite rank $n\in\N_+$, then every $\phi\in\End(G)$ is pointwise algebraic. Indeed, each $\phi\in\End(G)$ can be expressed by a matrix $A\in M_n(\Q)$. Let $p(t)\in\Q[t]$ be the characteristic polynomial of $A$. After multiplication by an integer we may assume that $p(t) \in \Z[t]$ (although it will not be monic any more). Then $p(\phi)(x) = 0$ holds for every  $x \in G$. Consequently, $G=\mathfrak W(G,\phi)$ for every $\phi\in\End(G)$.
\item[(c)] If $\phi\in\End(\Q)$, then there exists $r\in\Q$ such that $\phi(x)=r x$ for every $x\in \Q$. Hence, $$\mathfrak A(\Q,\phi)=\begin{cases}0 & \text{if}\ r\not\in\Z, \\ \Q & \text{if}\ r\in\Z; \end{cases}\ \ \text{and}\ \ \QQ(\Q,\phi)=\begin{cases}0 & \text{if}\ r\not\in\{0,1\}, \\ \Q & \text{if}\ r\in\{0,1\}. \end{cases}$$
So,
$$
0=\QQ(\Q,\phi)\subset \mathfrak A(\Q,\phi)=\Q=\mathfrak W(\Q,\phi)\mbox{ if }r\in\Z\setminus\{0,1\}.
$$ 
On the other hand, 
$$
\QQ(\Q,\phi)=0=\mathfrak A(\Q,\phi)\subset\Q=\mathfrak W(\Q,\phi)\mbox{ if }r\not\in\Z.$$
\end{itemize}
\end{example}


\subsection{Passing to $\mod_{R[t]}$}\label{4.2}

We pass now to $\mod_{R[t]}$  (through the isomorphism given by Theorem \ref{iso}) and its radicals $\QQ, \mathfrak O, \mathfrak I$, $\mathfrak A$ and $\mathfrak W$.

\begin{remark}\label{remark}
It is easy to see that if $M\in\mod_{R[t]}$, then:
\begin{itemize}
\item[(a)] $x\in\QQ(M)$ if and only if there exist $k\in\N_+$ and pairs of naturals $(n_j,m_j)$ with $n_j>m_j$ for $j=0,\ldots,k-1$ and such that $(t^{n_0}-t^{m_0})\cdot\ldots\cdot (t^{n_{k-1}}-t^{m_{k-1}})\cdot x=0$;
\item[(b)] $x\in\mathfrak O(M)$ if and only if there exists $k\in\N$ such that $t^k\cdot x=0$;
\item[(c)] $x\in\mathfrak I(M)$ if and only if there exists $k\in\N$ such that $(t-1)^k\cdot x=0$;
\item[(d)] $x\in\mathfrak A(M)$ if and only if there exists a monic $p(t)\in R[t]$ such that $p(t)\cdot x=0$;
\item[(e)] $x\in\mathfrak W(M)$ if and only if there exists $p(t)\in R[t]\setminus\{0\}$ such that $p(t)\cdot x=0$.
\end{itemize}
\end{remark}

This can be expressed by radicals of $\mod_{R[t]}$  (see Corollary \ref{uff}). We give first the following general result: 

\begin{theorem}\label{tts}
Let $A$ be a ring and $S\subseteq A$. For every left $A$-module $M$, let $$\r_S(M)=\{x\in M: (\exists s\in S)sx=0\}.$$ 
If $S$ is multiplicatively closed (i.e., $s',s \in S$ yields $s's\in S$), then $\r_S:\mod_A\to\mod_A$ is a hereditary radical. In particular, $\mathfrak t_{\r_S}=(\mathcal T_{\r_S},\mathcal F_{\r_S})$ is a hereditary torsion theory in $\mod_A$.
\end{theorem}

\begin{proof}
That $\r_S$ is hereditary follows directly from the definitions. To check that $\r_S$ is a radical take an $A$-module $M$ and $x\in M$ such that the coset $x+\r_S(M)\in M/\r_S(M)$ belongs to $\r_S(M/\r_S(M))$. Then $sx\in\r_S(M)$ for some $s\in S$ by the definition of $\r_S$. So $s'(sx)\in\r_S(M)$ for some $s'\in S$. This yields $x\in\r_S(M)$, as $s,s'\in S$ and $S\subseteq A$ is multiplicatively closed. 
\end{proof}

Theorem \ref{tts} together with Remark \ref{remark} give immediately the following result. Note that $R[t]\setminus\{0\}$ is multiplicatively closed if and only if $R$ is an integral domain.

\begin{corollary}\label{uff}
\begin{itemize}
\item[(a)] Let $Q$ be the multiplicative closure of the subset $\{t^n-t^m:n,m\in\N,n>m\}$ of $R[t]$. Then $\QQ=\r_Q$.
\item[(b)] Let $O=\{t^n:n\in\N\}\subseteq R[t]$. Then $\mathfrak O=\r_O$.
\item[(c)] Let $I=\{(t-1)^n:n\in\N\}\subseteq R[t]$. Then $\mathfrak I=\r_I$.
\item[(d)] Let $A=\{p(t)\in R[t]:p(t)\ \text{monic}\}\subseteq R[t]$. Then $\mathfrak A=\r_A$.
\item[(e)] Let $W=R[t]\setminus\{0\}$. Then $\mathfrak W=\r_W$.
\end{itemize}

\smallskip
\noindent In particular, $\QQ$, $\mathfrak O$, $\mathfrak I$ and $\mathfrak A$ are hereditary radicals of $\mod_{R[t]}$. If $R$ is an integral domain, then also $\mathfrak W$ is a hereditary radical of $\mod_{R[t]}$.
\end{corollary}

Both $O$ and $I$ are contained in $Q$, and both $O$ and $I$ generate $Q$ as a multiplicative closed subset of $R[t]$. Moreover, $Q$ is contained in $A$ and $A$ in $W$. Then 
$$
\mathfrak O\leq \QQ,\ \mathfrak I\leq\QQ\ \text{and}\ \mathfrak Q\leq\mathfrak A;
$$
if $R$ is an integral domain, also $\mathfrak A\leq\mathfrak W$. Furthermore, $M = \mathfrak W (M)$ precisely when $M$ is torsion as an $R[t]$-module. Alternatively, 
$M \not = \mathfrak W (M)$ precisely when $M$ contains a copy of the cyclic module $R[t]$ (in terms of $\af_R$, when $M$ contains a copy of the Bernoulli shift $\beta_R$, see the proof of Theorem \ref{PvsW} for the relative argument).

\medskip
Corollary \ref{uff} and Theorem \ref{tts} give immediately the following

\begin{corollary}
The pairs $\mathfrak t_\QQ=(\mathcal T_\QQ,\mathcal F_\QQ)$, $\mathfrak t_\mathfrak O=(\mathcal T_\mathfrak O,\mathcal F_\mathfrak O)$, $\mathfrak t_\mathfrak I=(\mathcal T_\mathfrak I,\mathcal F_\mathfrak I)$ and $\mathfrak t_\mathfrak A=(\mathcal T_\mathfrak A,\mathcal F_\mathfrak  A)$ are hereditary torsion theories in $\mod_{R[t]}$. If $R$ is an integral domain, then also $\mathfrak t_\mathfrak W=(\mathcal T_\mathfrak W,\mathcal F_\mathfrak  W)$ is a hereditary torsion theory in $\mod_{R[t]}$.
\end{corollary}

By Theorem \ref{hb<->tt} there exist binary entropy functions $h_\QQ$, $h_\mathfrak O$, $h_\mathfrak I$ and $h_\mathfrak A$ in $\HB(\mod_{R[t]})$ which correspond respectively to these torsion theories, that is, $\mathfrak t_{h_\QQ}=\mathfrak t_\QQ$, $\mathfrak t_{h_\mathfrak O}=\mathfrak t_\mathfrak O$, $\mathfrak t_{h_\mathfrak I}=\mathfrak t_\mathfrak I$ and $\mathfrak t_{h_\mathfrak A}=\mathfrak t_\mathfrak A$. If $R$ is an integral domain, then there exists also $h_\mathfrak W\in\HB(\mod_{R[t]})$ such that $\mathfrak t_{h_\mathfrak W}=\mathfrak t_\mathfrak W$.


\subsection{Radicals and axioms}

We use now the radicals defined above to clarify the relations among some of the axioms introduced in Section \ref{ax-sec}. The following lemma follows from Corollary \ref{Ph=Phb} and gives various equivalent forms of (A0$_0$) and (A0$_1$) in terms of the radicals $\mathfrak O$, $\mathfrak I$ and the Pinsker radical $\P_h$.
Recall that the binary entropy functions $h_\mathfrak O$ and $h_\mathfrak I$, introduced at the very end of Section \ref{4.2}, satisfy $\P_{h_\mathfrak O}=\mathfrak O$ and $\P_{h_\mathfrak I}=\mathfrak I$. 

\begin{lemma}\label{OI-P}
Let $h\in\mathfrak H(\mod_{R[t]})$. The following conditions are equivalent:
\begin{itemize}
\item[(a)] $h$ satisfies \emph{(A0$_0$)} \emph{(}respectively, \emph{(A0$_1$)}\emph{)};
\item[(b)] $h^b$ satisfies \emph{(A0$_0$)} \emph{(}respectively, \emph{(A0$_1$)}\emph{)};
\item[(c)] $\mathfrak O\leq\P_h= \P_{h^b}$ \emph{(}respectively, $\mathfrak I\leq\P_h=\P_{h^b}$\emph{)};
\item[(d)] $h_\mathfrak O\prec h^b$ \emph{(}respectively, $h_\mathfrak I\prec h^b$\emph{)}.
\end{itemize}
\end{lemma}

In the case when $R$ is a field, using the equivalences of Lemma \ref{OI-P}, we can see that (A0$_1$) and (A0$_0$) are independent: 

\begin{example}
Assume that $R$ is a field. Then 
\begin{equation}\label{FIELD}
\mathfrak O(M)\cap\mathfrak I(M)=0\mbox{ for every }M\in\mod_{R[t]}.
\end{equation}
Indeed, let $x\in \mathfrak O(M)\cap\mathfrak I(M)$. Then there exist $n,m\in\N_+$ such that $t^nx=0$ and $(t-1)^mx=0$. Since $t^n$ and $(t-1)^m$ are coprime elements of $R[t]$, there exist $u(t),v(t)\in R[t]$ such that $1=u(t)t^n+v(t)(t-1)^m$. Therefore, $x=u(t)t^nx+v(t)(t-1)^mx=0$.
\begin{itemize}
\item[(a)] To prove that (A0$_1$) does not imply (A0$_0$) consider the entropy function $h_\mathfrak I$. Since $\P_{h_\mathfrak I}=\mathfrak I$, we conclude with Lemma \ref{OI-P} that $h_\mathfrak I$ satisfies (A0$_1$). Assume for a contradiction that $h_\mathfrak I$ satisfies also (A0$_0$). Then by Lemma \ref{OI-P} $h_\mathfrak O\prec h_\mathfrak I$. By Theorem \ref{tt} this yields $\mathfrak t_{h_\mathfrak O}\leq\mathfrak t_{h_\mathfrak I}$. Since $\P_{h_\mathfrak O}=\mathfrak O$ and $\P_{h_\mathfrak I}=\mathfrak I$, it follows that $\mathfrak O\geq\mathfrak I$, which contradicts \eqref{FIELD}.
\item[(b)] To show that (A0$_0$) does not imply (A0$_1$), argue as in (a) exchanging the roles of $\mathfrak O$ and $\mathfrak I$.
\end{itemize}
\end{example}

\begin{remark}\label{A6-h} According to \cite{DG1}, $\P_{h_a}(G,\phi)=\QQ(G,\phi)$ for every $(G,\phi)\in\af_\abg$, i.e.,  $\P_{h_a}=\QQ$ in $\mod_{\Z[t]}$. In particular, 
$\mathfrak t_{h_a}=\mathfrak t_\QQ$, that is, $\mathcal T_{h_a}=\mathcal T_\QQ\mbox{ and }\mathcal F_{h_a}=\mathcal F_\QQ$.
\end{remark}

We shall see now that the inequality $\QQ\leq\P_h$ remains true for a general entropy function $h$ 
satisfying the axioms (A0) and (A4$^*$). For the proof we need
the next lemma establishing that every finitely generated submodule of $ \mathfrak A(M,\phi)$ is contained in a finitely generated $\phi$-invariant submodule of $ \mathfrak A(M,\phi)$.
 
\begin{lemma}\label{LAAAAST_Lemma}
If $(M,\phi)\in\af_R$ and $N\in\mathcal F( \mathfrak A(M,\phi))$ then $T(\phi,N)\in\mathcal F(\mathfrak A(M,\phi))$. In particular, $\phi\restriction_{T(\phi,N)}$ is quasi-periodic. 
\end{lemma}
\begin{proof}
Since $N$ is finitely generated, there exists $n\in\N_+$ such that $\phi^{n+1}(N)\subseteq T_n(\phi,N)$. Consequently, $T(\phi,N)=T_n(\phi,N)$ is finitely generated, and $T(\phi,N)\in\mathcal F(\mathfrak A(M,\phi))$ as $\mathfrak A(M,\phi)$ is a $\phi$-invariant submodule of $M$.
Finally, $\phi\restriction_{T(\phi,N)}$ is quasi-periodic, as $T(\phi,N)\in\mathcal F(Q_1(M,\phi))$. 
\end{proof}

\begin{theorem}\label{semiA6}
Let $h$ be an entropy function of $\af_R$ satisfying \emph{(A0)} and \emph{(A4$^*$)}. Then $\QQ\leq\P_h$.
\end{theorem}
\begin{proof}
Let $(M,\phi)\in\af_R$.
We show first that
\begin{equation}\label{qp->h=0}
\phi\ \text{quasi-periodic}\ \Longrightarrow\ h(\phi)=0.
\end{equation}
Since $\phi$ is quasi-periodic, the semigroup $\{\phi^n:n\in\N_+\}$ is finite and so it contains an idempotent, say $\phi^k$ for some $k\in\N_+$. Since $h(\phi^k)=kh(\phi)$ by (A4$^*$), to prove \eqref{qp->h=0} we can assume without loss of generality that $\phi$ is idempotent. Under this assumption, let $N=\ker\phi$. Then $\phi\restriction_N=0_N$ and so $h(\phi\restriction_N)=0$ by (A0$_0$). Moreover, let $\overline\phi:G/N\to G/N$ be the endomorphism induced by $\phi$. Since $G/N\cong \phi(M)$, $\overline\phi$ is conjugated to $\phi\restriction_{\phi(M)}$. Now $\phi\restriction_{\phi(M)}=1_{\phi(M)}$, so $h(\phi\restriction_{\phi(M)})=0$ by (A0$_1$), and hence $h(\overline\phi)=0$ by (A1). Finally (A2) gives $h(\phi)=0$. This concludes the proof of \eqref{qp->h=0}.

\smallskip
We prove by induction that 
\begin{equation}\label{ff}
h(\phi\restriction_{Q_n(M,\phi)})=0\ \text{for every}\ n\in\N.
\end{equation} 
This is obvious for $n=0$. Let $n=1$ and suppose $N\in\mathcal F(Q_1(M,\phi))$.  
Then $\phi\restriction_{T(\phi,N)}$ is quasi-periodic by Lemma \ref{LAAAAST_Lemma}. Thus $h(\phi\restriction_{T(\phi,N)})=0$ by \eqref{qp->h=0}. Again by Lemma \ref{LAAAAST_Lemma}, $Q_1(M,\phi)$ is a direct limit of finitely generated $\phi$-invariant subgroups, so (A3) implies $h(\phi\restriction_{Q_1(M,\phi)})=0$ in view of Remark \ref{A3d}.

Assume now that $h(\phi\restriction_{Q_n(M,\phi)})=0$ for some $n\in\N$, and let $L=Q_{n+1}(M,\phi)$. Consider the endomorphism $\overline\phi:L/Q_n(M,\phi)\to L/Q_n(M,\phi)$ induced by $\phi$. By Lemma \ref{Q/Q}(a),  $L/Q_n(M,\phi)=Q_1(L/Q_n(M,\phi),\overline\phi)$ by the choice of $L$. The case $n=1$ gives $h(\overline\phi)=0$. Then $h(\phi)=0$ by (A2). This concludes the proof of \eqref{ff}.

\smallskip
Hence, $h(\phi\restriction_{\QQ(M,\phi)})=0$ by \eqref{ff} and by the equivalent form of (A3) given in Remark \ref{A3d}.
\end{proof}

 It is proved in \cite{DG1} that the converse inequality $\QQ\geq\P_h$ of that in Theorem \ref{semiA6} holds in the case of $\af_\abg$ and $h_a$ (see Remark \ref{A6-h}), but its proof required the use of the Algebraic Yuzvinski Formula (see Theorem \ref{Yuz}). More precisely, this fundamental tool was shown to be necessary in the proof of the inclusion 
 $\QQ(\Q^d, \phi)\geq\P_h(\Q^d, \phi)$, for arbitrary $d\in \N$ and $\phi$. A proof of this inclusion, making no recourse to the Algebraic Yuzvinski Formula, 
 was obtained in the recent \cite{DGZ}. On the other hand, the proof given in \cite{DG1} of the inequality $\QQ\leq\P_{h_a}$ for $\af_\abg$ makes no recurse to (A2), as it uses other results  in that specific case related to the properties of the endomorphisms whose trajectories have polynomial growth. Therefore, it is fair to say that the verification of (A2)
  in the case of $\af_\abg$ and $h_a$ makes no recourse to the Algebraic Yuzvinski Formula. 

\medskip
Here comes another general result: 

\begin{theorem}\label{PvsW}
Let $R$ be a domain and let $h\not \equiv 0$ be an entropy function of $\af_R$. Then $\P_h \leq \mathfrak W$.
\end{theorem}
\begin{proof}
Let us prove first that $h(\beta_R) > 0$. Assume for a contradiction that $h(\beta_R)=0$. If $(M,\phi)\in \af_R$, then for $x\in M$ we define a homomorphism
$f: (R^{(\N)},\beta_R) \to (M,\phi)$ in $\af_R$. For $n\in\N$, let $e_n$ be the $n$-th canonical generator of $R^{(\N)}$ (so that 
$e_{n+1}= \beta(e_n)$). Then the map defined by $f(e_n) = \phi^n(x)$ for $n\in \N$ can be extended to a homomorphism with the desired
properties. Clearly, $f(R^{(\N)})= T(\phi, Rx)$. Hence, our hypothesis $h(\beta_R) = 0$ yields that $h(\phi\restriction_{T(\phi, Rx)})=0$. 
Let $N \in \mathcal F(M)$. Then $T(\phi,N)$ is a finite sum of $\phi$-invariant submodules of the form $N_x= T(\phi, Rx)$, with $x\in N$. 
Since $h(\phi\restriction_{T(\phi, Rx)})=0$ for each $x\in N$, we deduce that also $h(\phi\restriction_{N})=0$. In view of Remark \ref{A3d}, (A3) gives $h(\phi) = 0$. Hence, $h \equiv 0$, a contradiction. 

Assume that  $\P_h (M,\phi) \not\subseteq \mathfrak W(M,\phi)$  for some $(M,\phi)\in \af_R$  and pick $x\in \P_h (M,\phi)\setminus \mathfrak W(N,\phi\restriction_N)$. 
Then $R\phi^n(x) \cap (Rx + \ldots + R\phi^{n-1}(x))= 0$ for every $n\in \N_+$. Hence, the submodule $T(\phi, Rx)$ is isomorphic
to $R^{(\N)}$. Fix an isomorphism $f: R^{(\N)} \to T(\phi, Rx)$ as above. Then $f$ is also a morphism in $\af_R$, so $\beta_R$ is conjugated to
$\phi\restriction_{T(\phi, Rx)}$. 
Since $T(\phi, Rx)$ is the smallest $\phi$-invariant submodule containing $x\in \P_h (M,\phi)$ and the latter is a $\phi$-invariant submodule, 
we conclude that $T(\phi, Rx)\leq {\P_h (M,\phi)}$. Hence,  $ h(\phi\restriction_{T(\phi, Rx)}) = h(\phi\restriction_{\P_h (M,\phi)})=0$ yields
$h(\beta_R)= 0$ as well. By the above argument this leads to a contradiction. 
\end{proof}

\begin{corollary}\label{Corollary:Ps_vs_W}
Let $R$ be an integral domain and $h\not \equiv 0$ an entropy function of $\af_R$
 satisfying \emph{(A0)} and \emph{(A4$^*$)}. Then $\QQ\leq\P_h  \leq \mathfrak W$.
\end{corollary}

\section{Examples}\label{ex-sec}

In this section we consider the general results of the previous sections with respect to the known algebraic entropies.

\subsection{Algebraic entropy in $\abg$}\label{ent-sec}

We start with our principal and motivating example, that is, the algebraic entropy $h_a$ of $\af_\abg$. 

\smallskip
Let $(G,\phi)\in\af_\abg$. For a non-empty finite subset $F$ of $G$ and for any positive integer $n$, let $\tau_{\phi,F}(n)=|T_n(\phi,F)|.$ 
Then the limit
$$H(\phi,F)=\lim_{n\to\infty}\frac{\log\tau_{\phi,F}(n)}{n}$$ exists, as proved in \cite{DG}, and it is called the \emph{algebraic entropy of $\phi$ with respect to $F$}. The \emph{algebraic entropy} of $\phi$ is 
\begin{equation}\label{h-pet}
h_a(\phi)=\sup\{H(\phi,F): F\subseteq G\ \text{non-empty, finite}\}.
\end{equation}

It is clear from the definition that $h_a(0_G)=0$ for every abelian group $G$ and it is proved in \cite{DG1} that $h_a(1_G)=0$ for every abelian group $G$. In other words, $h_a$ satisfies (A0).

\medskip
In the next fact we collect the basic properties of the algebraic entropy proved in \cite{DG} (see also \cite{Pet}).

\begin{fact}\label{properties}
Let $(G,\phi)\in\af_\abg$.
\begin{itemize}
\item[(a)]If $(H,\psi)\in\af_\abg$ and $\phi$ and $\psi$ are conjugated, 
then $h_a(\phi)=h_a(\psi)$.
\item[(b)] If $k\in\N$, then $h_a(\phi^k) = k\cdot h_a(\phi)$. If $\phi$ is an automorphism, then $h_a(\phi^k) = |k|\cdot h_a(\phi)$ for every $k\in \Z$.
\item[(c)] If $G$ is a direct limit of $\phi$-invariant subgroups $\{G_j:j\in J\}$, then $h_a(\phi)=\sup_{j\in J}h_a(\phi\restriction_{G_j})$.
\item[(d)] For every prime $p$, $h_a(\beta_{\Z(p)})=\log p$.
\end{itemize}
\end{fact}

This fact implies that $h_a$ satisfies (A1), (A4$^*$), (A3) and (A5) (as a consequence of items (a), (b), (c) and (d), respectively).

\medskip
The following theorem is one of the main results on the algebraic entropy $h_a$ proved in \cite{DG}.

\begin{theorem}[Addition Theorem]\label{AT}
Let $G$ be an abelian group, $\f\in\End(G)$, $H$ a $\phi$-invariant subgroup of $G$ and $\overline\phi:G/H\to G/H$ the endomorphism induced by $\phi$. Then $h_a(\f)=h_a(\f\restriction_H) + h_a(\overline\f).$
\end{theorem}

The Addition Theorem shows exactly that $h_a$ satisfies (A2$^*$). It implies in particular that for $(G,\phi)\in\af_\abg$, the algebraic entropy $h_a$ is monotone under taking restrictions to $\phi$-invariant subgroups $H$ of $G$ and under taking endomorphisms $\overline\phi$ induced by $\phi$ on quotients $G/H$, that is, $h_a(\phi)\geq\max\{h_a(\phi\restriction_H),h_a(\overline\phi)\}$.

\medskip
Therefore, the algebraic entropy $h_a$ is an example of entropy function in the sense of Definition \ref{h-def}:

\begin{theorem}\label{tt-ha}
The algebraic entropy $h_a:\af_\abg\to\R_+\cup\{\infty\}$ is an entropy function of $\af_\abg$.
\end{theorem}

In this particular case Theorem \ref{tt} gives that $\mathfrak t_{h_a}=(\mathcal T_{h_a},\mathcal F_{h_a})$ is a hereditary torsion theory in $\af_\abg$, recalling that $$\mathcal T_{h_a}=\{(G,\phi)\in\af_\abg: h_a(\phi)=0\}\ \text{and}\ \mathcal F_{h_a}=\{(G,\phi)\in\af_\abg: h_a(\phi)>\!\!>0\}.$$ 

\medskip
Let $f(t)=a_nt^n+a_1t^{n-1}+\ldots+a_0\in\Z[t]$ be a primitive polynomial. Let $\{\lambda_i:i=1,\ldots,n\}\subseteq\mathbb C$ be the set of all roots of $f(t)$.
The \emph{(logarithmic) Mahler measure} of $f(t)$ is $$m(f(t))= \log|a_0| + \sum_{|\lambda_i|>1}\log |\lambda_i|.$$
The Mahler measure plays an important role in number theory and arithmetic geometry and is involved in the famous Lehmer's Problem, asking whether $\inf\{m(f(t)):f(t)\in\Z[t]\ \text{primitive}, m(f(t))>0\}>0$ (for example see \cite{Ward0} and \cite{Hi}).
If $g(t)\in\Q[t]$ is monic, then there exists a smallest positive integer $s$ such that $sg(t)\in\Z[t]$; in particular, $sg(t)$ is primitive. The Mahler measure of $g(t)$ is defined as $m(g(t))=m(sg(t))$. Moreover, if $\phi:\Q^n\to \Q^n$ is an endomorphism, its characteristic polynomial $p_\phi(t)\in\Q[t]$ is monic, and we can define the Mahler measure of $\phi$ as $m(\phi)=m(p_\phi(t))$.

\medskip
The following Algebraic Yuzvinski Formula shows that for an endomorphism of $\Q^n$ the algebraic entropy coincides with the Mahler measure. A direct proof of the Algebraic Yuzvinski Formula is given in \cite{Simone2} in the particular case of endomorphisms of $\Z^n$  and in \cite{GV} in the general case of endomorphisms of $\Q^n$. It is deduced in \cite{DG1} from the Yuzvinski Formula for the topological entropy of automorphisms of the Pontryagin dual $\widehat\Q^n$ of $\Q^n$ (see \cite{LW,WardLN,Y}) and from the ``Bridge Theorem'' proved by Peters in \cite{Pet} (see Theorem \ref{Peters} below). 

\begin{theorem}[Algebraic Yuzvinski Formula]\label{Yuz}
For $n\in \N_+$, if $\phi:\Q^n\to\Q^n$ is an endomorphism, then $h(\f)=m(\phi)$. 
\end{theorem}

The Algebraic Yuzvinski Formula was heavily used in the proof of the Addition Theorem \ref{AT} in \cite{DG}, and this was the reason to avoid the use of (A2) in the proof of $\QQ\leq\P_{h_a}$ in \cite{DG1}.
It plays an important role also in the following 

\begin{theorem}[Uniqueness Theorem]\label{UT}
The algebraic entropy $h_a$ of $\af_\abg$ is characterized as the unique collection $h_a=\{h_G:G\in\abg\}$ of functions $h_G:\End(G)\to\R_+\cup\{\infty\}$ such that:
\begin{itemize}
\item[(a)] $h_G$ is invariant under conjugation for every abelian group $G$;
\item[(b)] the Addition Theorem holds for $h_G$ for every abelian group $G$;
\item[(c)] if $\phi\in\End(G)$ and the abelian group $G$ is a direct limit of $\phi$-invariant subgroups $\{G_j:j\in J\}$, then $h_G(\phi)=\sup_{j\in J}h_{G_j}(\phi\restriction_{G_j})$;
\item[(d)] $h_{\Z(p)^{(\N)}}(\beta_{\Z(p)})=\log p$ for every prime $p$;
\item[(e)] the Algebraic Yuzvinski Formula holds for $h_\Q$ restricted to the automorphisms of $\Q$.
\end{itemize}
\end{theorem}

Note that the logarithmic law, which is axiom (A4$^*$) does not appear among the axioms needed in the Uniqueness Theorem, since it is embodied in (d) and (e).

\subsection{Algebraic entropy in $\mathbf{TorAbGrp}$}\label{entW-sec}

In Section \ref{ent-sec} we have discussed the algebraic entropy $h_a$ of $\af_\abg$, and we have shown that it is an entropy function. But, as said in the introduction, the algebraic entropy $h_a$ introduced by Peters in \cite{Pet} is a modification of the algebraic entropy $\ent$ by Weiss in \cite{W}. The algebraic entropy $\ent$ is defined as $h_a$, but taking in \eqref{ent-sec} the supremum with $F$ ranging among all finite subgroups of $G$, instead of all non-empty finite subsets of $G$. More precisely, for $(G,\phi)\in\af_\abg$,
$$\ent(\phi)=\sup\{H(\phi,F):F\leq G, F\ \text{finite}\}.$$
From the definition it follows directly that
\begin{equation}\label{h-ent}
\ent(\phi)=\ent(\phi\restriction_{t(G)})=h_a(\phi\restriction_{t(G)}),
\end{equation}
where $t(G)$ is the torsion subgroup of $G$.

\smallskip
Since $h_a$ satisfies (A0), \eqref{h-ent} implies immediately that also $\ent$ satisfies (A0).

\medskip
The following are the basic properties of the algebraic entropy $\ent$, proved in \cite{DGSZ,W}.

\begin{fact}\label{properties-ent}
Let $(G,\phi)\in\af_\abg$.
\begin{itemize}
\item[(a)] If $(H,\psi)\in\af_\abg$ and $\phi$ and $\psi$ are conjugated, then $\ent(\f) = \ent(\psi)$.
\item[(b)] For every $k\in\N$, $\ent(\f^k) = k \cdot \ent(\f)$. If $\f$ is an automorphism, then $\ent(\f^k)=|k|\cdot\ent(\f)$ for every $k\in\Z$.
\item[(c)] If $G$ is a direct limit of $\f$-invariant subgroups $\{G_j : j \in J\}$, then $\ent(\f)=\sup_{j\in J}\ent(\f\restriction_{G_j})$.
\item[(d)] For every prime $p$, $\ent(\beta_{\Z(p)})=\log p$.
\item[(e)] If $G$ is torsion and $H$ is a $\f$-invariant subgroup of $G$, then $\ent(\f)=\ent(\f\restriction_H)+\ent(\overline\f)$, where $\overline\f:G/H\to G/H$  is the endomorphism induced by $\f$.
\item[(f)] The algebraic entropy of the endomorphisms of the torsion abelian groups is characterized as the unique collection $h = \{h_G : G \text{ torsion abelian group}\}$ of functions $h_G:\End(G) \to \mathbb R_+\cup\{\infty\}$ that satisfy (a), (b), (c), (d) and (e).
\end{itemize}
\end{fact}

It follows respectively from (a), (b) and (c) of Fact \ref{properties-ent} that $\ent$ satisfies (A1), (A4$^*$) and (A3) considered on all $\af_\abg$. Moreover, (d) is (A5) and (e) shows that $\ent$ satisfies (A2$^*$) (and so (A2)) when is restricted to the subcategory $\af_\mathbf{TorAbGrp}$ of $\af_\abg$. Since $\mathbf{TorAbGrp}$ is an abelian category, $\af_{\mathbf{TorAbGrp}}$ is an abelian category as well.
So we have just verified the following

\begin{theorem}\label{tt-ent}
The algebraic entropy $\ent:\af_\mathbf{TorAbGrp}\to \R_+\cup\{\infty\}$ is an entropy function.
\end{theorem}

The same result can be obtained by Theorem \ref{tt-ha}, noting that $\ent$ coincides with $h_a$ on $\af_{\mathbf{TorAbGrp}}$ by \eqref{h-ent}.

\medskip
In view of Theorem \ref{tt-ent} the results of the previous sections can be applied to $\ent:\af_{\mathbf{TorAbGrp}}\to \R_+\cup\{\infty\}$. In particular, Theorem \ref{tt} implies that $\mathfrak t_\ent=(\mathcal T_\ent,\mathcal F_\ent)$ is a hereditary torsion theory in $\af_\mathbf{TorAbGrp}$, where 
$$\mathcal T_{\ent}=\{(G,\phi)\in\af_{\mathbf{TorAbGrp}}: \ent(\phi)=0\}\ \text{and}\ \mathcal F_{\ent}=\{(G,\phi)\in\af_\mathbf{TorAbGrp}:\ent(\phi)>\!\!>0\}.$$ 

Let us recall that for $(G,\phi)\in\af_\abg$ 
\begin{equation}\label{tphi}
t_\phi(G)=\{x\in G: |T(\phi,\langle x\rangle)|<\infty\}
\end{equation}
is the \emph{$\phi$-torsion} subgroup of $G$. 
For $(G,\phi)\in\af_{\mathbf{TorAbGrp}}$, from \cite{DG1} we have 
\begin{equation}\label{DG1}
\P_\ent(G,\phi)=t_\phi(G)=Q_1(G,\phi)=\QQ(G,\phi).
\end{equation}
 
In particular, considering the restriction $\QQ_T$ of the radical $\QQ$ to $\af_{\mathbf{TorAbGrp}}$, we have that $$\mathfrak t_\ent=\mathfrak t_{\QQ_T},$$ that is, the inequality of Theorem \ref{semiA6} becomes an equality in this case.

\subsection{$i$-Entropy in $\mod_R$}\label{ient-sec}

\subsubsection{The entropy associated to an additive invariant}

Let $R$ be a ring. Recall, that an invariant $i:\mod_R\to\R_+\cup\{\infty\}$  is additive if it satisfies (A2$^*$), and an additive invariant is a length function if it is also upper continuous (see Definition \ref{L-def} above).
Moreover, $i$ is \emph{discrete} if it has values in a subset of $\mathbb R_+$ order-isomorphic to $\N$ \cite{SZ}.

\begin{remark}
The algebraic entropy $\ent$ is discrete as it takes values in $\log\N_+\cup\{\infty\}$. 

On the other hand, it is known that the above mentioned Lehmer's Problem is equivalent to the problem of finding the value of $\inf\{h_a(\phi):(G,\phi)\in\af_\abg, h_a(\phi)>0\}$ (see \cite{WardLN}) --- note that this follows also from the Algebraic Yuzvinski Formula (i.e., Theorem \ref{Yuz}).  This value is positive if and only if  the algebraic entropy $h_a$ is discrete (see \cite{DG,DGZ}).
\end{remark}

For an invariant $i$ of $\mod_R$, the class of all $R$-modules $M$ with $i(M)<\infty $ is closed under finite sums and quotients. If $i$ is an additive invariant, then this class is closed also under submodules and extensions. Following \cite{SZ,SVV}, for $M\in\mod_R$ let $$\Fin_i(M)=\{N\leq M: i(N)<\infty\},$$
\begin{equation}\label{LAST-eq}
z_i(M)=\{x\in M:i(Rx)=0\} \mbox{ and }f_i(M)=\{x\in M:i(Rx)<\infty\}.
\end{equation}
For every $\phi\in\End(M)$, $z_i(M)$ is a $\phi$-invariant submodule of $M$. Moreover, $z_i:\mod_{R}\to \mod_{R}$ is a hereditary radical when $i$ is a length function, while the preradical $f_i:\mod_R\to \mod_R$ need not be a radical (see \cite[Example 2.5]{SVV}). 

\smallskip 
Following \cite{SVV,Simone}, for $M\in\mod_R$ we denote by $N_{i*}$ the $\mathfrak t_{z_i}$-closure of a submodule $N$ of $M$ (in the sense of Definition \ref{tr-closure} above), that is, $N_{i*}=\pi^{-1}(z_i(M/N))$ where $\pi:M\to M/N$ is the canonical projection. 

\medskip
If $f_i(M)=M$ we say that $M$ is \emph{locally $i$-finite}. 
Let $\mathrm{lFin}_i(R)$ be the class of all locally $i$-finite left $R$-modules. As noted in \cite{SVV} this class is closed under quotients, direct sums and submodules, while in general it is not closed under extensions. 

\begin{remark}\label{R-toriR}
As the ring $R$ is a generator of $\mod_R$, the behavior of an additive invariant $i$ of $\mod_R$ depends on whether $R\in \mathrm{lFin}_i(R)$ or not.
If $i(R)=0$, then $i(M)=0$ for every $M\in\mod_R$, that is, $i\equiv0$ (since every $M\in\mod_R$ is quotient of some free module $R^{(I)}$, which has $i(R^{(I)})=0$ by the 
additivity of $i$). In the sequel assume that $i\not \equiv0$, i.e., $i(R)>0$. 
\begin{itemize}
\item[(a)] If $R\not\in\mathrm{lFin}_i(R)$, then $i(R)=\infty$. In this case, if $x\in f_i(M)$, then $x$ is torsion, i.e., $\mathrm{ann}_R(x) \ne 0$. 

An example to this effect is given by $R=\Z$ and  $i=\log|-|$, since $\log|\Z|=\infty$. In this case $f_i(G)=t(G)$ for every abelian group $G$.

\item[(b)] If $R\in\mathrm{lFin}_i(R)$ (i.e., $0< i(R)< \infty $), then $\mathrm{lFin}_i(R) = \mod_R$ as $\mathrm{lFin}_i(R)$ is closed under direct sums  and quotients (nevertheless, $i(R^{(I)})=\infty$ if $I$ is infinite, by the additivity of $i$).  By Proposition \ref{i(R)_fin}, $i(R/I)=0$ for every proper ideal $I$ of $R$. 

An example to this effect is given by $i=r_0:\abg\to\N\cup\{\infty\}$; in fact, $r_0(\Z)=1$ and $r_0(\Z^{(\N)})$ is infinite.
\end{itemize}
\end{remark}

\medskip
In analogy with \eqref{h-ent} for the algebraic entropy $\ent$, also for the $i$-entropy of $(M,\phi)\in\af_R$ we have 
\begin{equation}\label{veryLAST-eq}
\ent_i(\phi)=\ent_i(\phi\restriction_{f_i(M)}) \text{ and }\ent_i(\phi\restriction_{z_i(M)})=0.
\end{equation}

\medskip
Let $(M,\phi)\in\af_R$ and consider an additive invariant $i$ of $\mod_R$. It is proved in \cite{SZ} that for $F\in\mathrm{Fin}_i(M)$ the limit $$H_i(\phi,F)=\lim_{n\to\infty}\frac{i(T_n(\f,F))}{n}$$ exists, and $H_i(\phi,F)$ is the \emph{algebraic entropy of $\f$ with respect to $F$}. The \emph{algebraic entropy of $\phi$} is $$\ent_i(\phi)=\sup\{H_i(\phi,F):F\in\Fin_i(M)\}.$$ 
Obviously, $i \equiv 0$ yields $\ent_i \equiv 0$. 

\begin{example}
For the invariant $i=\log|-|$ of $\abg$, $\ent_i$ is exactly the algebraic entropy $\ent$.
\end{example}

The following fact is a direct consequence of the definition (see also \cite[Proposition 1.8]{SZ}) and it shows that $\ent_i$ satisfies (A0).

\begin{fact}
Let $(M,\phi)\in\af_R$. If $F\in\Fin_i(M)$ is $\phi$-invariant, then $H(\phi,F)=0$. In particular, $\ent_i(0_M)=0$ and $\ent_i(1_M)=0$. 
\end{fact}

Many properties of the $i$-entropy were studied in \cite{SZ}. In particular, the following properties hold.

\begin{fact}\emph{\cite{SZ}}\label{properties-i}
Let $R$ be a ring and $i$ a discrete additive invariant of $\mod_R$. Let $(M,\phi)\in\af_R$.
\begin{itemize}
\item[(a)] If $(N,\psi)\in\af_R$ and $\f$ and $\psi$ are conjugated, then $\ent_i(\phi)=\ent_i(\eta)$.
\item[(b)] If $k\in\N$, then $\ent_i(\phi^k) = k\cdot \ent_i(\phi)$. If $\phi$ is an automorphism, then $\ent_i(\phi^k) = |k|\cdot \ent_i(\phi)$ for every $k\in \Z$.
\item[(c)] For every $M\in\mod_R$, $\ent_i(\beta_M)=i(M)$.
\item[(d)] If $N$ is a $\phi$-invariant submodule of $M$, then $\ent_i(\phi)\geq\ent_i(\phi\restriction_N)$.
\item[(e)] If $M=M_1\times M_2$ in $\mod_R$ and $\phi_j\in\End(M_j)$, for $j=1,2$, then $\ent_i(\phi_1\times\phi_2)=\ent_i(\phi_1)+\ent_i(\phi_2)$.
\item[(f)] \emph{\cite[Corollary 2.18(i)]{SVV}} Let $(M,\phi)\in\af_R$ and $N$ a $\phi$-invariant submodule of $M$. Then $\ent_i(\phi\restriction_N)=\ent_i(\phi\restriction_{N_{i*}})$.
\end{itemize}
\end{fact}

This fact implies that for an additive invariant $i$ of $\mod_R$, the function $\ent_i$ satisfies, beyond (A0), also (A1), (A4$^*$) and (A5) (as a consequence of items (a), (b) and (c), respectively). In Theorem \ref{enti-ef} we show that $\ent_i$ is an entropy function, but to this end we need to restrict appropriately its domain and impose two more conditions on $i$.

\subsubsection{The entropy function of a length function}

It is proved in \cite{SVV} that the Addition Theorem for $\ent_i$ holds for discrete length functions $i$ in the class of locally $i$-finite $R$-modules; in other words, $\ent_i$ restricted to $\af_{\mathrm{lFin}_i(R)}$ satisfies (A2$^*$).

\begin{theorem}[Addition Theorem]\label{AT-i}
Let $R$ be a ring and $i:\mod_R\to\R_+\cup\{\infty\}$ a discrete length function. Let $M\in\mod_R$ be locally $i$-finite, $\phi\in\End(M)$, $N$ a $\phi$-invariant submodule of $M$ and $\overline\phi:M/N\to M/N$ the endomorphism induced by $\phi$. Then $\ent_i(\phi)=\ent_i(\phi\restriction_N)+\ent_i(\overline\phi).$
\end{theorem}

Also an impressive Uniqueness Theorem is proved in \cite{SVV} for $\ent_i$, with $i$ discrete length function.
This theorem in particular generalizes the Uniqueness Theorem proved in \cite{SZ} for the entropy function $\ent_{r_0}$ of $\af_\abg$ and the Uniqueness Theorem for the algebraic entropy $\ent$ of $\af_\mathbf{TorAbGrp}$ (see Fact \ref{properties-ent}(f)).

\bigskip
Let $i$ be a discrete length function on $\mod_R$, and consider $\ent_i:\af_R\to\R_+\cup\{\infty\}$.
Since $\mathrm{lFin}_i(R)$ is a cocomplete abelian category, and $\af_{\mathrm{lFin}_i(R)}$ is a cocomplete abelian category as well, it is possible to consider $\ent_i$ restricted to $\af_{\mathrm{lFin}_i(R)}$ and to prove the following

\begin{theorem}\label{enti-ef}
Let $i$ be a discrete length function on $\mod_R$. Then $\ent_i:\af_{\mathrm{lFin}_i(R)}\to \R_+\cup\{\infty\}$ is an entropy function.
\end{theorem}
\begin{proof}
It remains to note that (A2) holds according to Theorem \ref{AT-i}, while (A3) follows from the fact (established in \cite{SVV}) that $\ent_i$ is an upper continuous invariant of $\mod_{R[t]}$ (see Remark \ref{A3d}).
\end{proof}

The restriction to the subcategory $\mathfrak M = \mathrm{lFin}_i(R)$ seems to be necessary, as $\ent_i$  may fail to satisfy (A2) (even monotonicity under taking induced endomorphisms on quotients) in $\mod_{R[t]}$. 

\medskip
In this particular case, for $$\mathcal Z_{\ent_i}=\{(M,\phi)\in\af_{\mathrm{lFin}_i(R)}: \ent_i(\phi)=0\}\ \text{and}\ \mathcal P_{\ent_i}=\{(M,\phi)\in\mathrm{lFin}_i(R):\ent_i(\phi)>\!\!>0\},$$  by Theorem \ref{tt}  we have that $\mathfrak t_{\ent_i}=(\mathcal Z_{\ent_i},\mathcal P_{\ent_i})$ is a hereditary torsion theory in $\mathrm{lFin}_i(R)$. Moreover, 
 $z_i(M)\subseteq \P_{\ent_i}(M,\phi)$ for every $(M,\phi)\in\af_{\mathrm{lFin}_i(R)}$.


\subsubsection{The Pinsker radical of the $i$-entropy}

Inspired by \eqref{tphi}, for $(M,\phi)\in\af_R$ define 
$$
t^i_\phi(M)=\{x\in M: i(T(\phi,Rx))<\infty\}.
$$ 
In particular,  $t^i_\phi(M)\subseteq f_i(M)$.

For $i=\log|-|$ and $(G,\phi)\in\af_\mathbf{TorAbGrp}$ the subgroup $t^i_\phi(G)$ coincides with the $\phi$-torsion subgroup $t_\phi(G)$  recalled in \eqref{tphi}.
Now we extend this property to $\ent_i$: 
 
\begin{theorem}\label{P=t}
Let $(M,\phi)\in\af_{\mathrm{lFin}_i(R)}$. Then $\P_{\ent_i}(M,\phi)=t^i_\phi(M)$.
\end{theorem}

\begin{proof} According to {\cite[Proposition 1.10(i)]{SZ}}, for $F\in\Fin_i(M)$, $H_i(\phi,F) = 0$ if and only if $i(T(\phi,F))<\infty$. Consequently, $\ent_i(\phi)=0$ if and only if $M=t^i_\phi(M)$.
In particular, $\ent_i(\phi\restriction_{t^i_\phi(M)})=0$, so 
$t^i_\phi(M)$ is the greatest $\phi$-invariant submodule of $M$ with this property.
Since this is also the defining property of  $\P_{\ent_i}(M,\phi)$ (see Lemma \ref{hP=0}),
we conclude that $\P_{\ent_i}(M,\phi)=t^i_\phi(M)$.
\end{proof}

For $(M,\phi)\in\af_{\mathrm{lFin}_i(R)}$, 
$\QQ(M,\phi)\subseteq \P_{\ent_i}(M,\phi)$ by Theorem \ref{semiA6}. This inclusion can be strict as Example \ref{example} will show.

\medskip
The next theorem improves Corollary \ref{Corollary:Ps_vs_W}.

\begin{theorem}\label{A<P}
Let $(M,\phi)\in\af_{\mathrm{lFin}_i(R)}$. Then $\mathfrak A(M,\phi)\subseteq \P_{\ent_i}(M,\phi)$.  
If $R$ is an integral domain, then $\P_{\ent_i}(M,\phi)\subseteq \mathfrak W(M,\phi)$.
\end{theorem}

\begin{proof} The inclusion $\P_{\ent_i}(M,\phi)\subseteq \mathfrak W(M,\phi)$ was proved in the general setting in Theorem \ref{PvsW}. 
The inclusion $\mathfrak A(M,\phi)\subseteq \P_{\ent_i}(M,\phi)$ follows substantially from \cite[Proposition 2.22]{Simone} (a proof appears also in \cite{DGV}). 
\end{proof}

Theorems \ref{P=t}, \ref{A<P} and \eqref{veryLAST-eq} give 
$$
z_i(M)+\mathfrak A(M,\phi)\subseteq \P_{\ent_i}(M,\phi)=t_\phi^i(M).
$$
As a consequence of Fact \ref{properties-i}(f), since both $z_i(M)$ and $\mathfrak A(M,\phi)$ are contained in  $\mathfrak A(M,\phi)_{i*}$, one can strengthen this as follows:  
\begin{corollary}
For every $(M,\phi)\in\af_R$, $\P_{\ent_i}(M,\phi)$ is $\mathfrak t_{z_i}$-closed in $M$. In particular, if $M\in\mathrm{lFin}_i(R)$, then $\mathfrak A(M,\phi)_{i*}\subseteq\P_{\ent_i}(M,\phi)$.
\end{corollary}


The next example shows that ${\mathfrak A(M,\phi)}_{i*}$ (and so also $\mathfrak A(M,\phi)$) may be strictly contained in $\P_{\ent_i}(M,\phi)$, even if $z_i(M)=0$.

\begin{example}\label{example}
Let $\phi\in\End(\Q)$, that is, there exists $r\in\Q$ such that $\phi(x)=r x$ for every $x\in\Q$. Since $\Q$ has finite free-rank, $\ent_{r_0}(\phi)=0$, and hence $\P_{\ent_{r_0}}(\Q,\phi)=\Q$. 
\begin{itemize}
\item[(a)] If $r\in\Z\setminus\{0,1\}$, then $\QQ(\Q,\phi)=0$ and $\mathfrak A(\Q,\phi)=\Q=\mathfrak W(\Q,\phi)$ by Example \ref{LAST-ex}. Hence, $0=\QQ(\Q,\phi)\subset\mathfrak A(\Q,\phi)=\mathfrak A(\Q,\phi)_{r_0*}=\mathfrak W(\Q,\phi)=\Q=\P_{\ent_{r_0}}(\Q,\phi)$.
\item[(b)] If $r\not\in\Z$, then $\mathfrak A(\Q,\phi)=\QQ(\Q,\phi)=0$ and $\mathfrak W(\Q,\phi)=\Q$ by Example \ref{LAST-ex}. 
Therefore, $0=\QQ(\Q,\phi)=\mathfrak A(\Q,\phi)=\mathfrak A(\Q,\phi)_{r_0*}\subset\mathfrak W(\Q,\phi)=\Q=\P_{\ent_{r_0}}(\Q,\phi)$.
\end{itemize}
\end{example}

In case $R$ is an integral domain and $i(R)<\infty$ we have the following 

\begin{example} {\rm \cite{DGV}}
Let $R$ be an integral domain and $i$ a length function on $\mod_R$ such that $i(R)<\infty$. Then $\mathrm{lFin}_i(R)=\mod_R$
 and $i(R/I)=0$ for every proper ideal $I$ of $R$ (see Remark \ref{R-toriR}(b)). Moreover, for every $(M,\phi)\in\af_{R}$ the following conditions are equivalent:
\begin{itemize}
\item[(a)]$\ent_i(\phi)=0$;
\item[(b)]$\phi$ is pointwise algebraic over $R$.
\end{itemize}
One can prove that under the above hypotheses, $i$ coincides with the multiple $i(R)\mathrm{rank}_R$ of the rank $\mathrm{rank}_R$ over $R$ (see \cite[Theorem 2]{NR}).
For $R= \Z$ this generalizes the equivalence established in \cite{SZ} for the rank-entropy in $\abg$.
\end{example}

\begin{corollary}\label{LASTcorollary}
Let $R$ be an integral domain and $i$ a length function on $\mod_R$ such that $i(R)<\infty$.
Then $\P_{\ent_i}=\mathfrak W$ in $\af_R$.
\end{corollary}

\section{Contravariant entropy functions}\label{contr-sec}

A measure space is a triple $(X,\mathfrak  B, \mu)$, where $X$ is a non-empty set, $\mathfrak  B$ a $\sigma$-algebra on $X$ and $\mu$ is a measure on $\mathfrak  B$. If $(X_i,{\mathfrak B}_i,\mu_i)$, $i=1,2$, are measure spaces,
then a map $T:(X_1,{\mathfrak B}_1,\mu_1)\to (X_2,{\mathfrak B}_2,\mu_2)$ is a {\em measure preserving transformation} if  $T^{-1}({\mathfrak B}_2)\subseteq {\mathfrak B}_1$
and  $\mu_1(T^{-1}(A))=\mu_2(A)$  for all $A\in {\mathfrak B}_2$. Denote by $\mathbf{Mes}$ the category of all measure spaces and their measure preserving transformations. The measure entropy $h_{mes}$ in $\mathbf{Mes}$ is defined as follows. 

For a measure space $(X,{\mathfrak B},\mu)$ and a measurable partition $\xi=\{A_1,A_2,\ldots,A_k\}$ of $X$ define the \emph{(measure) entropy} of $\xi$ by 
$$
H_{mes}(\xi)=-\sum_{i=1}^k \mu(A_k)\log \mu(A_k).
$$
For two partitions $\xi, \eta$ of $X$ let $\xi \vee \eta=\{U\cap V: U\in \xi, V\in \eta\}$. Analogously define $\xi_1\vee \xi_2\vee \ldots \vee \xi_n$ for partitions $\xi_1,\ldots,\xi_n$ of $X$.
For a measure preserving $T:X\to X$ and a measurable partition $\xi=\{A_1,A_2,\ldots,A_k\}$ of $X$, let $T^{-j}(\xi)=\{T^{-j}(A_i)\}_{i=1}^k$. Then the limit 
$$
h_{mes}(T,\xi)=\lim_{n\to \infty}\frac{H_{mes}(\bigvee_{j=0}^{n-1}T^{-j}(\xi))}{n}
$$ 
exists. The \emph{(measure) entropy} of $T$ is $h_{mes}(T)=\sup _\xi h_{mes}(T,\xi)$, where $\xi$ runs over the family of all measurable partitions of $X$. 
 
\medskip
The topological entropy $h_{top}$ in the category $\mathbf{Comp}$ of all compact spaces and their continuous maps is defined as follows.  For a compact topological space $X$ and for an open cover $\mathcal U$ of $X$, let $N(\mathcal U)$ be the minimal cardinality of a subcover of $\mathcal U$. Since $X$ is compact, $N(\mathcal U)$ is always finite. Let $H(\mathcal U)=\log N(\mathcal U)$ be the \emph{entropy of $\mathcal U$}.
For any two open covers $\mathcal U$ and $\mathcal V$ of $X$, let $\mathcal U\vee\mathcal V=\{U\cap V: U\in\mathcal U, V\in\mathcal V\}$. Define analogously $\mathcal U_1\vee\ldots	\vee\mathcal U_n$, for open covers $\mathcal U_1,\ldots,\mathcal U_n$ of $X$.
For a continuous map $\psi:X\to X$ and an open cover $\mathcal U$ of $X$ let $\psi^{-1}(\mathcal U)=\{\psi^{-1}(U):U\in\mathcal U\}$.
The \emph{topological entropy of $\psi$ with respect to $\mathcal U$} is $$H_{top}(\psi,\mathcal U)=\lim_{n\to\infty}\frac{H(\mathcal U\vee\psi^{-1}(\mathcal U)\vee\ldots\vee\psi^{-n+1}(\mathcal U))}{n},$$ and the \emph{topological entropy} of $\phi$ is $$h_{top}(\psi)=\sup\{H_{top}(\psi,\mathcal U):\mathcal U\ \text{open cover of $X$}\}.$$

A definition of the topological entropy $h_d(f)$ for a continuous self-map $f$ of a metrizable topological space $(X,d)$ was given by Bowen \cite{Bow}. In case $(X,d)$ is compact, one has $h_d(f) =h_{top}(f)$ \cite{WardLN}. 

\medskip
Let $\textbf{CompGrp}$ denote the category of all compact groups and their continuous homomorphisms, and let $U : \textbf{CompGrp} \to \textbf{Comp}$ be the obvious forgetful functor. On the other hand, every compact group $G$ has a (unique) invariant measure $\mu_G$ (i.e.,  that makes all translations in $G$ measure preserving), namely the Haar measure.
It was noticed by Halmos \cite{Halmos}, that a continuous homomorphism in $\textbf{CompGrp}$ is measure preserving precisely when it is surjective. 
The epimorphisms in $\textbf{CompGrp}$ are precisely the surjective continuous homomorphisms \cite{P}. So, denoting by $\textbf{CompGrp}_e$ the non-full subcategory  
of  $\textbf{CompGrp}$, having as morphisms all epimorphisms in  $\textbf{CompGrp}$, we obtain also an obvious forgetful functor $V: \textbf{CompGrp} \to \textbf{Mes}$, given in the following diagram where $i$ is the inclusion of $\textbf{CompGrp}_e$ in $\textbf{CompGrp}$ as a non-full subcategory: 

\begin{equation*}
\xymatrix{
\textbf{CompGrp}_e\ar[r]^{\; i}\ar[d]^V & \textbf{CompGrp}\ar[r]^{\;\;\;\;\;\; U }& 
 \textbf{Top} \\
\textbf{Mes}
}
\end{equation*}

A remarkable property of this triple is that for $(G,\phi)\in\af_{\textbf{CompGrp}_e}$ (i.e., with $\phi$ surjective), the measure entropy $h_{mes}(V\phi)$ coincides with the topological entropy $h_{top}(U\phi)$. This was established by Aoki \cite{A} in the case of automorphisms and by Stojanov \cite{S} in the general case. 

On the other hand, it is possible to reduce general continuous homomorphisms of compact groups to surjective ones. Indeed, every continuous endomorphism $\phi$ of a compact group $K$ admits a largest closed $\phi$-invariant subgroup $E_\phi(K)$ such that $\phi\restriction_{E_\phi(K)}:E_\phi(K)\to E_\phi(K)$ is surjective and $h_{top}(\phi\restriction_{E_\phi(K)})=h_{top}(\phi)$ (see \cite[Corollary 8.6.1]{Wa}).

For the sake of simplicity, we shall write $h_{top}(\phi)$ in place of $h_{top}(U\phi)$ in the sequel. 

\bigskip
Another important connection between the topological entropy and the algebraic entropy (exploiting the Pontryagin duality) is the following so-called Bridge Theorem due to Peters: 

\begin{theorem}\label{Peters}\emph{\cite{Pet}}
If $G$ is a countable abelian group and $\phi$ is an automorphism of $G$, then $h(\phi)=h_{top}(\widehat \phi)$, where $\widehat G$ is the Pontryagin dual of $G$ and $\widehat\phi:\widehat G\to\widehat G$ is the adjoint automorphism of $\phi$.
\end{theorem}

In \cite{DG} we generalize this theorem, proving it for endomorphisms of arbitrary abelian groups.

\medskip
It is known that $h_{top}$ is ``continuous" with respect to inverse limits when considered on $\mathbf{CompGrp}$. This shows a 
substantial difference
compared to  the entropy functions we considered in the previous sections as the algebraic entropy and the $i$-entropy for $\abg$ and $\mod_R$ 
respectively, as they are ``continuous" with respect to direct limits.

\medskip
This gives a good motivation to split the abstract notion of entropy functions in  {\em two dual notions}, say \emph{covariant entropy functions} (precisely those of Definition \ref{h-def}) and \emph{contravariant entropy functions} as $h_{top}$ on $\mathbf{CompGrp}$. 

While both the covariant and contravariant entropy functions must be invariant under conjugation, and satisfy the Addition Theorem (or some weaker version of the Addition Theorem), the ``continuity" property must be imposed in a selective way: the covariant entropy functions must be ``continuous" with respect to direct limits, while the contravariant entropy functions could be ``continuous" with respect to inverse limits (here we respect the already existing record on the topological and the measure entropy).

\begin{remark}
The distinction between both types of entropy is well visible also in the case of the ``normalization axiom" that imposes a specific value of the 
entropy function at the Bernoulli shifts. For $K\in\abg$, the left Bernoulli shift ${}_K\beta$ of $G= K^{(\N)}$ has algebraic entropy $0$ (as $(G, {}_K \beta)= \mathfrak O(G, {}_K \beta)$), and the right Bernoulli shift $\beta_K$ of $K^{(\N)}$ has algebraic entropy $\log |K|$. 
Conversely, for $K\in\mathbf{CompGrp}$, the left Bernoulli shift ${}_K\beta$ of $K^\N$ has topological entropy $\log |K|$, while the right Bernoulli shift $\beta_K$ of $K^\N$ has topological entropy $0$. 
\end{remark}


So, in view of the properties of the measure entropy and of the topological entropy discussed above, we introduce the contravariant entropy functions in the following way. 

\begin{definition}\label{hc-def}
Let $\mathfrak N$ be a complete abelian category. A \emph{contravariant entropy function} of $\mathfrak N$ is a an entropy function $h:\mathfrak N^{op}\to \R_+\cup\{\infty\}$, where $\mathfrak N^{op}$ is the opposite category of $\mathfrak N$.
\end{definition}

Note that a contravariant entropy function $h$ satisfies (A1) and (A2) of Definition \ref{h-def}, as the conditions are self-dual; so the difference between covariant and contravariant entropy functions is contained in (A3) and its ``opposite''. 

\medskip
Semiabelian categories, introduced  in \cite{JMT}, provide a nice generalization of abelian categories which reflects the properties of the categories of groups, 
rings and algebras and allows for a categorical approach  to radical theories. 
As shown recently in \cite{BC}, the category $\textbf{CompGrp}$ is semiabelian but not abelian. This suggests to generalize the setting at least to semiabelian categories to include the topological entropy at least for continuous endomorphisms of compact groups. So we leave open the following problem, noting that the theory of torsion theories has been extended 
in this 
context 
(see for example in \cite{CDT,JT}).

\begin{problem}
Develop the theory of covariant (respectively, contravariant) entropy functions in semiabelian cocomplete (respectively, complete) categories.
\end{problem}


\end{document}